\documentclass{article}

\usepackage[utf8]{inputenc}
\usepackage[T1]{fontenc}
\usepackage[english]{babel}
\usepackage{amsmath}
\usepackage{amsthm}
\usepackage{amsfonts}
\usepackage{amssymb}  
\usepackage{xcolor}
\usepackage[a4paper,left=3cm,right=3cm]{geometry}
\usepackage{graphicx}
\usepackage{biblatex}
\usepackage{csquotes}
\usepackage{tikz-cd}
\usepackage{refcount}
\usepackage{hyperref}

\usepackage{thmtools}

\theoremstyle{definition}

\newcounter{sharedcounter}[section]

\newtheorem{lemma}[sharedcounter]{Lemma}

\newtheorem{definition}[sharedcounter]{Definition}
\newtheorem{remark}[sharedcounter]{Remark}
\newtheorem{corollary}{Corollary}

\declaretheorem[name=Definition,numbered=no]{defres*}
\declaretheorem[name=Theorem]{theorem}
\declaretheorem[name=Theorem,numbered=no]{theorem*}
\declaretheorem[name=Proposition,sibling=sharedcounter]{proposition}

\DeclareMathOperator{\pic}{\text{Pic}}
\DeclareMathOperator{\ch}{ch}
\DeclareMathOperator{\rk}{rk}

\DeclareMathOperator{\td}{td}
\DeclareMathOperator{\tr}{Tr}
\newcommand{\pone}{\mathbb{C P}^1}
\newcommand{\chara}[2]{\begin{bmatrix} #1 \\ #2 \end{bmatrix}}

\newcommand{\ov}[1]{\overline{#1}}

\newcommand{\eq}[2][1]{
\begin{equation*}
	\addtolength{\fboxsep}{0pt}
	\begin{alignedat}{#1}
	#2
	\end{alignedat}
\end{equation*}
}

\newcommand{\eqtag}[3][1]{
\begin{equation}
   \label{#3}
	\begin{alignedat}{#1}
	#2
	\end{alignedat}
\end{equation}
}

\renewcommand{\leq}{\leqslant}
\renewcommand{\geq}{\geqslant}

\DeclareBibliographyDriver{book}{%
  \printnames{author}%
  \iffieldundef{title}{}{\addcomma\addspace\mkbibemph{\printfield{title}}}
  \iffieldundef{year}{}{\addcomma\addspace\printfield{year}}%
  \iffieldundef{pages}{}{\addcomma\addspace\printfield{pages}}%
  \finentry}

\DeclareFieldFormat[article]{journaltitle}{#1}

\DeclareBibliographyDriver{article}{%
  \printnames{author}%
  \addcomma\space
  \mkbibemph{\printfield{title}}%
  \addcomma\space
  \printfield{journaltitle}%
  \addcomma\space
  \printfield{volume}%
  \space
  \mkbibparens{\printfield{year}}%
  \addcomma\space
  \printfield{pages}%
  \iffieldundef{eprint}
    {}
    {\addcomma\space\printfield[eprint:arxiv]{eprint}}%
  \finentry
}

\begin{document}

\title{Chern classes of Laughlin bundles\\ on the quasihole moduli space}
\author{Florent Dupont, Semyon Klevtsov}
\date{}

\maketitle

\begin{center}
 {\it\small
\noindent
IRMA, CNRS, Université de Strasbourg, 7 r.~René Descartes, 67084 Strasbourg, France}
   
\end{center}

\begin{abstract}

    We study fractional quantum Hall states with quasihole excitations, on Riemann surfaces of arbitrary genus. For configurations with $m$ quasiholes we construct a vector bundle above the $m$-th symmetric power of the curve so that the fiber at a point $\lbrace w_1,\dots,w_m \rbrace$ corresponds to the state with quasiholes localized at these positions. We determine the Chern character of this bundle via the Grothendieck-Riemann-Roch theorem and show that in the completely filled state, i.e. when the number of particles is maximal, the vector bundle is compatible with the condition of projective flatness. Furthermore, we obtain a generalization of this result to the case of multiple layers and multiple quasihole types. In genus zero and one, we construct explicit wave-functions and verify that the curvature of the associated Chern connection reproduces the predicted Chern classes. The Chern classes obtained match, term by term, the predicted decomposition of the Berry phase under quasihole exchange, into an extensive Aharonov–Bohm contribution and a fractional statistical contribution.

\end{abstract}

\tableofcontents

\section{Introduction}

\subsection{Laughlin wave-functions and quasihole moduli space}

The fractional quantum Hall effect is thought to be due to strong interactions between charge carriers in the material, but solving analytically the multiparticle Hamiltonian is very difficult. Instead, in the planar case, Laughlin proposed the following ansatz for what the multiparticle wave-function should look like:

$$\Psi_b(z_1,...,z_n)=P(z_1,\dots,z_n)\prod_{1\leq \mu<\nu \leq n}(z_\mu-z_\nu)^b \cdot e^{-\frac{1}{4}\sum_{\mu=1}^n|z_\mu|^2}$$ for $z_1,\dots,z_n \in \mathbb{C}$, $b\in\mathbb N$ and where $P$ is a symmetric polynomial in variables $\{z_n\}$. The wave function $\Psi_b$ belongs to $L^2(\mathbb{C}^n, \prod_{\mu=1}^n d^2 z_\mu)$. For odd $b$, $\Psi$ is completely antisymmetric and is interpreted as a fermionic state, whereas for even $b$ it is completely symmetric and corresponds to a bosonic state.

One of the major successes of the Laughlin theory is the prediction of localized excitations, called quasiholes, which follows naturally from the ansatz above. 
Let $w_1,\dots,w_m \in \mathbb{C}$ and
\begin{align}\label{Psib}\nonumber
\Psi_{b,c}(z_1,...,z_n|w_1,\dots,w_m)&\\=&Q(z_1,\dots,z_n)\prod_{1\leq\mu\leq n}\prod_{1\leq\gamma\leq m}(z_\mu-w_\gamma)^c \prod_{1 \leq \mu<\nu \leq n}(z_\mu-z_\nu)^b \cdot e^{-\frac{1}{4}\sum_{\mu=1}^n|z_\mu|^2},
\end{align} 
where $Q$ is again a symmetric polynomial. The state above  belongs to the same Hilbert space $\Psi_{b,c}\in L^2(\mathbb{C}^n, \prod_{\mu=1}^n d^2 z_\mu)$ as before and corresponds to a wave-function where the positions of some of the zeroes of $P$ are fixed. These are the \textit{quasiholes}  localized at $w_1,\dots,w_m$ and their positions are considered as parameters, which can potentially be controlled externally and adiabatically. 

Suggested by Laughlin as particle-like excitations, the quasiholes possess some remarkable  properties. Their electric charge is in fact fractional and quantized in units $1/b$ \cite{Laughlin_1983} and the exchange statistics is fractional as well. Namely, the geometric phase $\varphi$ that the wave function \eqref{Psib} acquires in the adiabatic process of exchanging two quasiholes is predicted \cite{Halperin_1984, Arovas_Schrieffer_Wilczek_1984} to be the sum of two terms
\begin{align}\label{eq:phase}\nonumber
    \varphi&=\varphi_0+\Delta\varphi\\&=
    -\pi\langle n\rangle+\frac{\pi}b,
\end{align}
where the "extensive" part $\varphi_0$ scales with the average number of particles $\langle n\rangle$ enclosed in the closed loop formed by the exchange process. We also refer to \cite[Eq.\ (E.9), App. E]{Can} for the complete version of the formula above in the case of a Riemann surface. The statistical part $\Delta\varphi=\frac{\pi}b$ of the phase signifies the fact that the quasiholes behave not as bosons or fermions, but rather as particles of anyonic statistics predicted for two-dimensional quantum systems in \cite{Leinaas_Myrheim_1977, Wilczek_Shapere_1989}. The statistical phase acquired when one of the quasiholes is moved around another one in a full circle is $\frac{2\pi}b$, twice the phase of Eq.\ \eqref{eq:phase}.

Theoretical studies of anyonic excitations in the QHE, especially the braiding properties of their wave-functions, have led to the development of topological quantum field theories \cite{Wilczek_1982,Halperin_1984,Arovas_Schrieffer_Wilczek_1984,Simon_1983,Berry_1984,Witten1989,FG1990,Moore_Read_1991,Wen_1990,Nayak_Simon_Stern_Freedman_Sarma_2008}. Another direction is the emergence of the anyonic gas in the quantum Hall systems, see e.g. Ref.\ \cite{Lundholm_Rougerie_2016} for the recent work. For a review of history and recent developments in physics and mathematical studies of anyons, we refer the reader to Ref.\ \cite{Lundholm_2023}. 

The quasiholes are considered as indistinguishable particles. Their moduli space is thus the $m$-th symmetric power of $M$ the manifold where the particles live. The diagonal set, where any pair of particles collides, is usually excluded to avoid singular points. Hence $\mathcal M_m= (M^m\setminus \Delta_m)/ \mathfrak{S}_m$, where $\mathfrak{S}_m$ is the $m$th symmetric group and $\Delta_m$ is the "big diagonal", the set in $M^m$, where positions of any two particles coincide.

In this paper, we begin with the observation that for a compact Riemann surface $M=C$ 
the whole symmetric space $S^mC=C^m/ \mathfrak{S}_m$ is in fact a smooth complex manifold of dimension $m$. 
The wave functions $\Psi_{b}$ Eq.\ \eqref{Psib} analytically continue to the diagonal $\Delta_m$ in an obvious fashion, so that when two quasiholes of unit charge merge, the resulting quasihole has a charge two, etc. The holomorphic family of wave functions over the quasihole moduli space $S^mC$ in fact forms a holomorphic vector bundle, and our goal is to compute and interpret its Chern classes. Here we avoid the discussion of whether it is energetically possible to merge two or more quasiholes or split a quasihole of multiple charge, which is a completely legitimate question to pose. Our main point is rather the following. Analytically continuing the parameter space to include the diagonals allows for the computation of the Chern classes which in turn allows to infer the emergence of the anyonic statistics, in the off-diagonal case as well. 

In the quantum Hall effect, the Hall conductance is quantized because it corresponds to a topological invariant, namely the first Chern class of the vector bundle of ground states over the space of Aharonov–Bohm fluxes, divided by its rank, see \cite{Avron_Seiler_Zograf_1994,Thouless_Kohmoto_Nightingale_Den_Nijs_1982,Klein_Seiler_1990,Simon_1983}. For higher genus surfaces, the bundle of Laughlin states over the space of Aharonov–Bohm fluxes 
$\pic^d(C)$ possesses higher Chern classes, and its Chern character has been computed in \cite{Klevtsov_Zvonkine_2022,Klevtsov_Zvonkine_2025}, and before in Refs.\ \cite{Avron_Seiler_Zograf_1994,Klevtsov2016} for the integer case $b=1$. In \cite{Aldonza_Dupont_2025} the Chern character of multilayer quantum Hall states on higher genus surfaces was computed. This paper continues this line of thought by determining the Chern classes for the quasihole moduli space and suggesting their interpretation in light of anyonic statistics of quasiholes. 

In Section \ref{sec:Laughlin}, for each quasihole configuration $w\in S^m C$, we construct a line bundle $L_{b,c,w}$ over $S^nC$ whose sections are Laughlin wave-functions with localized quasiholes at positions given by $w$, particle-particle vanishing given by $b$ and particle-quasihole vanishing given by $c$. We prove that the family $w \to H^0(S^n C, L_{b,c,w})$ forms a vector bundle above $S^m C$, which we call the quasihole bundle. To construct it explicitly, we first build a line bundle above $S^nC \times S^m C$ such that its restriction to $S^nC \times \lbrace w \rbrace$ is $L_{b,c,w}$. We define the quasihole bundle as a pushforward onto $S^mC$ of this universal bundle. We compute its Chern character through a Grothendieck-Riemann-Roch computation. In Section \ref{sec:explicit}, we write explicit wave-functions in the case $C=\pone$ as well as when $C$ is an elliptic curve and give explicit choices of metrics on sections. Using the Chern connection, we compute in another way the complete Chern character of the quasihole bundle.

\subsection{Multilayer wave-functions}

Next we consider the quasihole bundle in the multi-layer case, with multiple quasihole types.
The multi-layer states were  introduced by Halperin \cite{Halperin_1984} as a generalization of the Laughlin case, where the underlying surface comes in $k$ layer states. 

Let $K\in M^{k\times k}(\mathbb{N})$ be a symmetric matrix, and $k$ an integer. Consider the wave-functions of the form  
\begin{equation}\label{eq:Halperin_wavefunctions}
    \Psi_K(\lbrace z^i_\mu \rbrace)=P(\lbrace z_\mu^i \rbrace )\prod_{1\leq i\neq  j \leq k}\prod_{\substack{1\leq \mu \leq n_i \\ 1 \leq \nu \leq n_j}} (z_\mu^i-z_\nu^j)^{K_{ij}}\prod_{1\leq i \leq k}\prod_{1\leq \mu<\nu \leq n_i}(z_\mu^i-z_\nu^i)^{K_{ii}}\cdot e^{-\frac{1}{4}\sum_{\mu, i}|z_\mu^i|^2}.
\end{equation}

This wave-function corresponds to a setting where particles are distributed among $k$ layers, where layer $i$ has $n_i$ particles, and $z_\mu^i$ is the position of a particle labeled as $\mu$ that lives in layer $i$.  The interactions between the particles in one layer are encoded by the vanishing order $K_{ii}$, and between the particles in different layers $i$ and $j$ by $K_{ij}$, for $j=1,\dots,i-1,i+1,\dots,k$. In the equation above, $P$ is a polynomial that is symmetric under permutation layer by layer, i.e. under exchanges $z^i_\mu \leftrightarrow z^i_\nu$.

In this setting, one can also introduce localized quasiholes. A quasihole $w_\mu^s$ is characterized by its type $s$ to which is associated the column vector $\lbrace C_{is},1\leq i \leq k \rbrace$. For $i$ and $s$ fixed, $C_{is}$ is the order of vanishing of the wave-function when a particle in layer $i$ approaches the position $w_\mu^s$ of a quasihole of type $s$. For $q$ different quasihole types, these vanishing data form a matrix $C\in M^{k\times q}(\mathbb{N})$ and the wave function reads
\begin{equation}\label{eq:Halperin_wavefunctionsqh}
    \Psi_{K,C} (\{ z^i_\mu \} | \{ w^s_\gamma \} )=Q(\lbrace z_\mu^i \rbrace )\prod_{\substack{i,s\\ \mu,\gamma}} (z_\mu^i-w_\gamma^s)^{C_{is}} \prod_{\substack{i\neq  j\\\mu,\nu}}(z_\mu^i-z_\nu^j)^{K_{ij}}\prod_{\substack{i\\\mu<\nu}}(z_\mu^i-z_\nu^i)^{K_{ii}}\cdot e^{-\frac{1}{4}\sum_{\mu, i}|z_\mu^i|^2}
\end{equation}
with $Q$ again a polynomial symmetric under $z^i_\mu \leftrightarrow z^i_\nu$.

In section \ref{sec:multilayer} we compute the Chern character of the quasihole bundle in this case.

\subsection{Main results}

For the single-layer Laughlin quasihole bundle, our main results are as follows. Let $$
V:=V_{b,c,d,g,n,m}$$ be the holomorphic vector bundle over the quasihole moduli space $S^mC $, indexed by the vanishing order $b$, the quasihole vanishing order $c$, the degree $d$ of the underlying magnetic line bundle $L$ on the genus-$g$ Riemann surface with $n$ particles and $m$ quasiholes.

The Chern character of the Laughlin quasihole bundle is given by

\begin{restatable}[Chern character in the general case]{theorem}{mainqh}\label{thm:mainqh}
    $$\ch(V)=e^{-c n \xi_m} \sum_{j=0}^g \sum_{k=j}^{g} \binom{n-g+p}{k-g+p}\binom{g-j}{k-j} b^{k-j} \frac{(-c^2\theta_m)^j}{j!} $$
where $p=d-bn-cm-b(g-1)$.
\end{restatable}
Here, $\xi_m$ and $\theta_m$ are the standard cohomology classes in $H^2(S^mC,\mathbb Z)\cap H^{1,1}(S^mC,\mathbb C)$. Here and throughout the article, we use the convention that $\binom{n}{p}$ is zero whenever $p<0$ or $p>n$. In particular, it follows that for $p<0$, $\ch(V)=0$, i.e. there are no Laughlin states with quasiholes at all. The value $p=0$ corresponds to the maximal number of particles $n=n_{\rm max}$ such that Laughlin states exist, in which case we have $$d=bn_{\rm max}+cm+b(g-1)$$ for given $d$. This is the so-called "completely filled" configuration, in the sense that adding another particle becomes impossible. In this case, the formula above simplifies.

\begin{restatable}[Chern character in the completely filled case]{theorem}{mainqhff}
\label{thm:mainqhff}
    Suppose $d=bn+cm+b(g-1)$, then
    $$\ch(V)=b^g e^{-\frac{c^2}{b}\theta_m - c n \xi_m}.$$
\end{restatable}
This corresponds to the formula of Theorem \ref{thm:mainqh} with $p=0$.

Our next result is an explicit calculation of the expression above for the Chern character in the completely filled case via the Berry curvature in the low-genus cases. In genus \(g=0\) the explicit form of (holomorphic part of) the Laughlin state in the presence of quasiholes is the single holomorphic section
\[
\prod_{1\leqslant\mu<\nu\leqslant n} (z_\mu - z_\nu)^b \prod_{\substack{1\leqslant\mu\leqslant n\\1\leqslant\gamma\leqslant m}}(z_\mu-w_\gamma).
\]

In genus \(g=1\), the space of holomorphic states is \(b\)-dimensional, and one can choose a basis of the form
\begin{align*}
s_l(.|w):&\; C^n\mapsto H^0(C^n, L^{\boxtimes n})\\ & (z_1,\dots,z_n) \mapsto \; \theta \chara{\frac{l}{b}}{0} \left(\sum_{\mu=1}^n b z_\mu + \sum_{\gamma=1}^m w_\gamma,\,b\tau\right)
\prod_{1\leqslant\mu<\nu\leqslant n} \theta(z_\mu - z_\nu,\tau)^b
\prod_{\substack{1\leqslant\mu\leqslant n\\1\leqslant\gamma\leqslant m}}\theta(z_\mu-w_\gamma,\tau)
\end{align*}
for $0 \leq l \leq b-1$.

Viewed as functions of quasihole coordinates $\{w_\gamma\}$, these sections provide an explicit holomorphic frame of a rank $b^g$ vector bundle $V$ over the parameter space $S^mC$. The hermitian structure on this vector bundle is given by the natural choice of $L^2$ structure on the space of these sections above $S^nC$. There is a canonical Chern connection in the hermitian holomorphic vector bundle with the connection form given by $\omega=H^{-1}\partial_m H$, where $\partial_m=\sum_\gamma\partial_{w_\gamma}\wedge dw_\gamma$ see \cite[Ch.1]{Kobayashi}, where $H$ is
the Gram matrix $H=\langle s_i(.|w) ,s_j(.|w)\rangle_{1\leq i , j \leq \rk(V)}$ of $L^2$ products. These are given by the following $2n$ dimensional integrals. For the sphere, the $L^2$ product is given by Eq.\ \eqref{eq:L2sphere}:
\begin{equation*}
\langle s(.|w),s(.|w) \rangle_V=\int_{S^n C}\prod_{\mu<\nu} |z_\mu - z_\nu |^{2b} \prod_{\mu,\gamma}  |z_\mu-w_\gamma|^2 \prod_\mu h(z_\mu)^d d{\rm Vol}_n(z),
\end{equation*}
where the hermitian metric in the fibers is the round metric $h(z)=\frac1{1+|z|^2}$. 
For the torus, the $L^2$ product is given by Eq. \eqref{eq:L2torus},
$$\langle s(.|w),s'(.|w) \rangle_V = \int_{S^n C} s(z|w) \ov{s'(z|w)} e^{-\frac{2\pi d}{\Im(\tau)}\sum_\mu \Im(z_\mu)^2} \prod_\mu d{\rm Vol}(z).$$
This Gram matrix may be too hard to compute explicitly in any meaningful way, for finite $n$. However, things simplify if one is only interested in computing the curvature $\Omega=\bar\partial( H^{-1}\partial H)\in \Omega^{1,1}({\rm End}(V))$ of the Chern connection, and simplify even further if one is only interested in the De Rham cohomology classes given by $[\tr \Omega^k],\;0\leqslant k\leqslant m$, that is, $\tr \Omega^k$ up to an exterior derivative of a smooth $(2k-1)$-form.

The latter is what is necessary for the computation of the Chern character for $V$, via the formula $\ch (V)=[\tr e^{\frac{i}{2\pi}\Omega}]$, where $\Omega= \ov{\partial}_m H^{-1}\partial_m H$ is the curvature of the Chern connection $\omega$. Our computation confirms the result of Thm. \ref{thm:mainqhff}.

\begin{restatable}[Cohomology classes obtained from the Chern connection]{proposition}{propclasschernconnection}
\label{prop:classchernconnection}
   For Laughlin states on the torus with $m$ quasiholes of charge $c=1$
    $$[\tr e^{\frac{i}{2\pi}\Omega}]=be^{-\frac{1}{b}\theta_m-n\xi_m}$$
\end{restatable}

Here we add a few comments on the interpretation of these results.
The rank of the Laughlin quasihole bundle $\rk(V)=b^g$ in the completely filled case corresponds to the Wen-Niu \cite{Wen_Niu_1990} "topological degeneracy" on genus-$g$ surfaces, proved in \cite{Klevtsov_Zvonkine_2025}. The expression in the exponent corresponds to the first Chern class of the bundle divided by its rank
$$
\frac{c_1(V)}{\rk (V)}=-\frac{c^2}{b}\theta_m - c n \xi_m.
$$
In general, one cannot compute the holonomy of a connection in the vector bundle just from its first Chern class, since the latter is defined only in cohomology. However, in the case at hand the correspondence of two terms with the prediction of \eqref{eq:phase} is rather striking.  The theta-class encodes the fractional statistics part of the phase. We interpret this correspondence in the following way. Following e.g. \cite[p.52]{Witten_2016}, consider the adiabatic process on the torus, where first one of the quasiholes is transported around one of the non-trivial cycles, then another one is transported around the dual cycle and vice versa. The accumulated phase difference between the two states is given exactly by $\frac{2\pi}b\int\theta=\frac{2\pi}b$, where the integral is taken over the Jacobian variety of the torus, isomorphic to the torus itself. Under this process the holonomy of the wave function corresponds in fact to the integral of the curvature of the adiabatic connection over the full moduli space (Jacobian torus) and is thus completely controlled by the first Chern class.

The second term is the extensive part proportional to the number of particles $n$ and the charge $c$ of the quasihole as in the geometric phase. The extensive part of the first Chern class has been effectively computed in \cite{Read_2008} for the Laughlin quasiholes on the sphere.

Finally, we note that for the topological quantum computation using anyons \cite{Freedman_Kitaev_Larsen_Wang_2002,Kitaev_2003} it is important that the braiding of the quasiholes is independent of the geometry of the braiding path and only depends on the homotopy class of the path, possibly up to an overall $U(1)$ factor, see e.g. \cite{Read_2009,Nayak_Simon_Stern_Freedman_Sarma_2008} for discussion. This is the defining property of projectively flat vector bundles, where the holonomies of the corresponding connection depend only on $\pi_1(M)$, up to an overall phase. It is well-known that projectively flat bundles must have a Chern character of the form 
$$\ch(V)=\rk (V)\;e^{\frac{c_1(V)}{\rk(V)}},$$
which is exactly the form obtained in Thm.\ \ref{thm:mainqhff} in the completely filled state. 

It is quite hard to construct a unitary projectively flat connection in fractional quantum Hall states, except for some highly symmetric cases such as the torus. 
In Ref.\ \cite{Klevtsov_Zvonkine_2022} the exponential form of the Chern character as above was suggested as a "geometric test" for topological states of matter in the sense that the vector bundles whose Chern character is not exponential do not correspond to a topological state of matter. Note that this form does not hold in the case $p>0$, as in this case the state is not expected to be a topological state of matter.

For the multilayer case with multiple quasiholes, with particle-particle interaction $K$ and particle-quasihole interaction matrix $C$, we have the following result:

\begin{restatable}[Chern character in the multiparticles with multiple quasihole type setting]{theorem}{mainqhmulti}
\label{thm:mainqhmulti}
Suppose that the bilinear form associated to $K-I$ is non-negative and $$\vec{d}=K\vec{n}+C\vec{m}+\vec{K}(g-1).$$ Then 
\eq{
    \ch(V_{K,C,d,g,\vec{n},\vec{m}})=\det(K)^g \exp\left(|(-C^T K^{-1} C).\Theta_m| - \vec{n}^T C \xi_m\right)
    }
where $\xi_m$ and $\Theta_m$ are matrices of cohomology classes defined in Eq. \eqref{eq:bigtheta} and Eq. \eqref{eq:xivect}, and $|A|$ the sum of all the coefficients of a matrix $A$. The dot product $.$ denotes the Hadamard product: for two matrices $A=(A_{ij})_{1\leq i,j\leq n}$ and $B=(B_{ij})_{1\leq i,j\leq n}$, $A.B$ is defined by $(A.B)_{ij}=A_{ij}B_{ij}$. \end{restatable}

{\bf Acknowledgments.} We would like to thank Mar\'ia Abad Aldonza, Igor Burban and Dimitri Zvonkine for collaboration on \cite{Aldonza_Dupont_2025,Klevtsov_Zvonkine_2022,Klevtsov_Zvonkine_2025,Burban_Klevtsov_2025,burban2024norms}, and for useful discussions on the  algebro-geometric  approach to quantum Hall states. The second author was partly supported by the ANR-20-CE40-0017 grant, the Initiative d’Excellence program and the Institute for Advanced Study Fellowship of the University of Strasbourg.

\section{Laughlin wave-functions with localized quasiholes}
\label{sec:Laughlin}

\subsection{Geometric setting}

We consider the setting in which the particles are trapped on a Riemann surface $C$ of genus $g$ (or multiple copies of $C$). The magnetic field is represented by a $U(1)$-principal bundle $P\to C$ equipped with a connection $\nabla$.

\begin{proposition}
The associated magnetic line bundle $L=P\times_{U(1)} \mathbb{C}$ inherits a canonical hermitian metric, connection, and holomorphic structure. Because of the way they are obtained, those three structures are compatible in the sense that the connection is the Chern connection associated with the metric and the holomorphic structure.

\end{proposition}

\begin{proof}
Let $P$ be a $U(1)$-bundle with connection form $A \in \Omega^1(P,u(1))$, which is a $U(1)$-equivariant form on $P$. Recall that $P \times_{U(1)} \mathbb{C}$ is defined as $P\times \mathbb{C}$ quotiented by $(p,z)\sim(g.p,g^{-1} z)$ for any $g \in U(1)$. We get a hermitian metric $h$ by defining $h([p,z],[p,z'])= z \ov{z'}$.

$L$ lives above a complex manifold, so the connection $\nabla$ splits as a holomorphic and an anti-holomorphic part: $\nabla=\nabla^{1,0}+\nabla^{0,1}$. Because $C$ is 1-dimensional, we automatically have $({\nabla^{0,1}})^2=0$, thus $\nabla^{0,1}$ defines a Dolbeault operator and gives a holomorphic structure to $L$.

To see that the connection is the Chern connection associated to $h$ and $\ov{\partial}$, it remains to check that the connection is unitary. Let $\sigma$ be a section of $L$. Pick a local trivialization of $P$ given by $s:U \subset C \to P$. Locally, $\sigma=[s(x),f(x)]$ and the connection on $L$ is given locally by $\nabla=d+s^\star A$.

Since $(\nabla s,s)_L=df \ov{f} + s^\star A f \ov{f}$ and $(s,\nabla s)_L=df \ov{f} - s^\star A f \ov{f}$ where the minus sign comes from the fact that $s^\star A$ is purely imaginary  as a form valued in the Lie algebra of $U(1)$, we get $$d(s,s)_L=df \ov{f} +  f d\ov{f}=(\nabla s, s)_L+(s,\nabla s)_L.$$

We are going to use this line bundle $L$ to construct a line bundle whose sections  are quantum states with quasiholes.

\end{proof}

\subsection{Construction of the quasihole bundle}

Denote by $n$ the number of particles. We can build a line bundle on the Cartesian product $C^{n}$ as $L^{\boxtimes n}=\bigotimes_\mu \pi_\mu^*L$ where $\pi_\mu:C^n\rightarrow C$ denotes the projection over the $\mu$-th factor. Let $m\in \mathbb{N}$ and $(w_1,\dots,w_m)$ a point in $C^m$.
We will call quantum Hall states with localized quasiholes at $w_1,\dots,w_m$ the sections of $L^{\boxtimes n}$ that satisfy the following properties:

\begin{itemize}
    \item Their restriction to each copy of $C$ gives a holomorphic section of $L$.
    \item They are symmetric (resp. antisymmetric) under exchange of two particles when $b$ is even (resp. odd)
    \item They vanish at order $b \in \mathbb{N}$ when two particles are in the same position.
    \item  They vanish at order $c \in \mathbb{N}$ when one of the particles is at position $w_l$ for some $l$.
\end{itemize}

These axioms define on a Riemann surface the states introduced by Laughlin for the complex plane, and are motivated by the following physics considerations. The first axiom encodes the holomorphicity of the ground state as a holomorphic section of $L$ for any individual particle. With this construction, the restriction of a multiparticle section to a single copy of $C$ is indeed a holomorphic section of $L$.
The second axiom ensures that the particles are indistinguishable, and the third and fourth axioms model the particle-particle and particle-quasihole interaction. In other words the wave-functions behave locally as holomorphic functions having a factor $$\prod_{1\leqslant\mu<\nu\leqslant n}(z_\mu-z_\nu)^b \prod_{\substack{1\leq \mu \leq n\\1\leq \gamma \leq m}} (z_\mu - w_\gamma)^c.$$
The latter two conditions are local in nature and thus are exactly the same as in the original formulation by Laughlin and the first condition takes care of global properties, when passing to a compact surface.

\begin{proposition}
Let $(w_1,\dots,w_m) \in C^m$ and $w$ be the image of this point in $S^m C$. There is a line bundle $L_{b,c,w }\to 
S^n C$ such that $H^0(S^{n}C,L_{b,c,w})$ is in bijection with the vector space of Laughlin states with quasiholes at positions $w_1\dots,w_m$.
\end{proposition}

\begin{proof}
Let $\Delta_n$ be the divisor in $C^n$ corresponding to the big diagonal, i.e. 
$$\Delta_n=\cup_{\mu<\nu} \lbrace (z_1,\dots,z_n) \in C^n,  z_\mu=z_\nu \rbrace$$ and let 
$$W_{(w_1,\dots,w_m)}=\cup_{\substack{1\leq \mu \leq n\\1\leq \gamma \leq m}} \lbrace z_\mu=w_\gamma \rbrace$$ be the quasihole divisor on $C^n$. Note that $W_{(w_1,\dots,w_m)}$ is invariant under permutation of the $m$ quasihole positions, so we denote it by $$W_w=W_{(w_1,\dots,w_m)}.$$
To impose the vanishing conditions, we consider $L^{\boxtimes n}(-b\Delta_n - c W_{w})$. The sections of this line bundle correspond to sections of $L^{\boxtimes n}$ that vanish at order $b$ when $z_\mu=z_\nu$ for some $\mu \neq \nu$ and at order $c$ when $z_\mu=w_\gamma$ for some $1\leq \mu \leq n,1\leq \gamma\leq m$. 

Since $L^{\boxtimes n}$, $\Delta_n$ as well as $W_{w}$ are invariant under the action of the permutation group with $n$ elements $\mathfrak{S}_n$, $L^{\boxtimes n}(-b\Delta_n - c W_{w})$ descends to a unique line bundle $L_{b,c,w}$ on $S^n C$. By construction, quantum Hall states with localized quasiholes at $\{w_1,\dots,w_m\}$ are holomorphic sections of $L_{b,c,w}$. 
\end{proof}

These form a family of states parameterized by $S^m C$ which will form a vector bundle. We will construct this vector bundle $V_{b,c}$ as a pushforward of a universal bundle $\mathcal{L}_{b,c}$ above $S^n C \times S^m C$ which we define now.

\begin{proposition}
There exists a unique line bundle (up to isomorphism) $\mathcal{L}_{b,c} \to S^n C \times S^m C$ such that for all $w$ in $S^m C$, 
$$\mathcal{L}_{b,c}|_{S^nC \times \lbrace w \rbrace}\simeq L_{b,c,w}$$
and 
$$\mathcal{L}_{b,c}|_{\lbrace nz_0 \rbrace \times S^mC} \simeq O(-c n Q)$$
where $Q=\{w=w_1+\dots+w_m,z_0\in \{w_1,\dots,w_m\}\}$ and where the last isomorphism is an isomorphism of topological line bundles.

\end{proposition}
    
\begin{proof}
The bundle is uniquely determined by these two conditions (by the Seesaw theorem, see \cite[\S 2, Corollary 6]{Mumford_1988}), so we only need to prove existence. Let $$\Delta_{n,m}=\cup_{\substack{1\leq \mu \leq n\\1\leq \gamma \leq m}} \lbrace z_\mu = w_\gamma \rbrace \subset C^{n+m}$$ and $\text{pr}_1: C^{n+m} \to  C^n$. Since $\text{pr}_1^\star L^{\boxtimes n}$,$\text{pr}_1^\star \Delta_n$ and $\Delta_{n,m}$ are invariant under the action of $\mathfrak{S}_n \times \mathfrak{S}_m$ where each group acts respectively on the first $n$ variables and last $m$ variables, the line bundle $\text{pr}_1^\star (L^{\boxtimes n}(-b\Delta_n)) \otimes O(-c\Delta_{n,m})$ descends to a line bundle $\mathcal{L}_{b,c}$ on $S^n C\times S^mC$ which satisfies the two conditions.
\end{proof}

\begin{definition}
Let $\pi:S^nC \times S^m C \to S^m C$.  We define $$V_{b,c,d,g,n,m}:=\pi_\star \mathcal{L}_{b,c}.$$ We are going to prove that $V_{b,c,d,g,n,m}$ is indeed a vector bundle. This is a vector bundle of Laughlin states over the quasihole moduli space $S^m C$, which we call for short the Laughlin quasihole bundle.
\end{definition}

\subsection{Cohomology classes in products of symmetric powers and computation of $c_1(\mathcal{L}_{b,c})$}

In this section, we introduce the cohomology classes that we are going to need for the Grothendieck-Riemann-Roch computation. We will need to compute the first Chern class of $\mathcal{L}_{b,c}$, which is an element of the Neron-Severi group of the base space.

We take $n>2g-1$ and $m>2g-1$ and fix some $z_0 \in C$.
For a generic curve $C$, $NS(S^n C)$ is known. Indeed, the map $S^n C \to \pic^nC$ is the projectivization of a vector bundle $E$ fixed up to isomorphism. $NS(S^n C)$  contains two classes: $\theta_n$ the pullback of the theta divisor from $\pic^n(C)$, and $\xi_n$  defined as $c_1(\mathcal{O}_{\mathbb{P}(E)}(1))$. In cohomology, they are linked by a relation: 

\[
        H^\bullet(S^n C,\mathbb{Z})
        \;\cong\;
        H^\bullet(\pic^n C,\mathbb{Z})[\xi_n]
        \Big/
        \bigl(\xi_n^{n-g+1} -\theta_n \,\xi_n^{n-g} + \cdots + (-1)^g \frac{\theta_n^g}{g!} \xi_n^{n-2g+1}\bigr).
\]
see \cite[Ch. 7]{Griffiths_book}.
Note that $\xi_n$ admits another representation in $H^2(S^n C,\mathbb{Z})$ as follows: for any $z_0 \in C$, we have
\begin{equation}
\label{eq:xisecondresentation}
    \xi_n=\lbrace z_1+\dots+z_n \in S^n C, z_0 \in \lbrace z_1,\dots,z_n\rbrace \rbrace
\end{equation}

We write $\theta_m$ and $\xi_m$ the respective classes on $S^mC$.

By pulling back those classes, we obtain four classes in $NS(S^n C \times S^m C)$, which we write the same way by abuse of notation. 
Another way to obtain classes in $NS(S^n C \times S^m C)$ is to pull back the theta divisor $\theta_{n+m}$ of $S^{n+m} C$ through the map $\sigma: S^n C \times S^m C \to S^{n+m} C$. In  doing so, we get a new class 
\eqtag{
    \eta_{n,m}=\sigma^\star \theta_{n+m} -\theta_n- \theta_m
}{eq:pullbacktheta} 
 that is mixed. We will express later this class using pullbacks of a symplectic basis of degree 1 cohomology on $\pic^{n}(C)$, which will show that this class is non-zero. 
 
\begin{remark}
Doing the same procedure for $\xi_{n+m}$ on $S^{n+m} C$ does not give a new class, as $\sigma^\star \xi_{n+m}-\xi_n-\xi_m=0$. This can be proven using the description of Eq. \eqref{eq:xisecondresentation} for $\xi_m$ from which it follows that as sets  $$\sigma^{-1} \left(\lbrace w \in S^{n+m} C, z_0 \in w \rbrace\right)=\left(\lbrace w \in S^{n} C, z_0 \in w \rbrace \times S^m C \right) \cup \left(S^n C \times \lbrace w \in S^{n+m} C, z_0 \in w \rbrace\right)$$ which becomes $$\xi_{n+m}=\xi_n+\xi_m$$ in cohomology (and in Chow for the same choice of $z_0$).

 \end{remark}

\begin{lemma}
\label{lem:expincohomologyofbigdiag}
Let $S\Delta_n$ (resp. $S\Delta_{n,m}$) be the image in $S^n C\times S^mC$ of $\Delta_n$ (resp. $\Delta_{n,m}$) in $C^{n+m}$. Then, in the cohomology ring of $S^n C\times S^m C$, 

\begin{equation}
\label{eq:diagonal}
S\Delta_{n}=2(-\theta_n+(n+g-1) \xi_n). 
\end{equation}

\begin{equation}
    \label{eq:mixedDiagonal}
S\Delta_{n,m}=-\eta_{n,m}+m\xi_n+n\xi_m
\end{equation}

\end{lemma}

\begin{proof} 
For a proof of Eq. \eqref{eq:diagonal} see \cite{Griffiths_book}. Eq. \eqref{eq:mixedDiagonal} was proven in \cite{Aldonza_Dupont_2025} and is easily obtained from \eqref{eq:diagonal} with a small transversal multiplicity computation, for the convenience of the reader, we recall the proof.
    
Let $\sigma: S^{n}C \times S^{m}C\rightarrow S^{n+m}C$ the forgetful map, as before. As sets, $\sigma^{-1} \left( S\Delta_{n+m}\right)=S\Delta_n \cup S\Delta_m \cup S\Delta_{n,m}$ but as classes, 
    
    \begin{equation}
    \label{eq:transversal_mult}
    \sigma^\star S\Delta_{n+m}=S\Delta_n+S\Delta_m+2S\Delta_{n,m}
    \end{equation}
   
Indeed, let $(\{z_1, z_2, \dots\},\{z_1, w_2, \dots\})$ be a generic point in $S\Delta_{n,m}$, with  a small neighborhood $U$. $\sigma:U\to \sigma(U) $ is a branched covering of degree two: a perturbation along $S\Delta_{n+m}$ keeps the number of antecedents to one, but a perturbation $\{z_1, z_2, \dots,z_{n},z_1,w_2,\dots,w_{m}\} \to \{z_1, z_2, \dots,z_{n},z_1',w_2,\dots,w_{m}\}$ outside of $S\Delta_{n+m}$ makes the number of antecedents jump to two: $(\{z_1, z_2, \dots\},\{z_1', w_2, \dots\})$ and $(\{z_1', z_2, \dots\},\{z_1, w_2, \dots\})$.

   Writing \eqref{eq:transversal_mult} in cohomology in terms of the classes we introduced reads 
   \begin{align*}
   2 S\Delta_{n,m}= &\sigma^\star ( 2(-\theta_{n+m}+(n+m+g-1) \xi_{n+m}))\\& - 2(-\theta_n+(n+g-1) \xi_n) - 2(-\theta_m+(m+g-1) \xi_m)
   \end{align*}
   Since there is no two-torsion in $NS(S^n C \times S^m C)$, we can divide by two on both sides. 
   The result follows after using the expression \eqref{eq:pullbacktheta} to get 
   \eq{
    \sigma^\star ((-\theta_{n+m}+(n+m+g-1) \xi_{n+m}))=-\theta_n-\theta_m-\eta_{n,m}+(n+m+g-1)(\xi_n+\xi_m).
   }

\end{proof}

\begin{proposition}
\label{prop:firstcherclassunivquasiholebundle}
    $$c_1(\mathcal{L}_{b,c})=b\theta_n + c \eta_{n,m} + p \xi_n  -cn\xi_m$$ where $p=d-bn-cm-b(g-1)$.
\end{proposition}

\begin{proof}
Write $D$ a divisor on $S^nC \times S^m C$ such that its pullback to $C^{n+m}$ is $b\Delta_n+c\Delta_{n,m}$. In cohomology, $D=b \frac{1}{2}S\Delta_n+c S\Delta_{n,m}$. The descent of $L^{\boxtimes n}$ has first Chern class $d\xi_n$. The result is then obtained with
\eq{c_1(\mathcal{L}_{b,c})&=d\xi_n-c_1(D)\\
    &=d\xi_n - \frac{1}{2}b c_1(S\Delta_n)-c c_1(S\Delta_{n,m})
    } by using the lemma \ref{lem:expincohomologyofbigdiag}.
\end{proof}

\subsection{Applying the Grothendieck-Riemann-Roch theorem}
\label{subsec:applyinggrr}

\subsubsection{The case with no non-localized quasiholes}

Let
$$\pi: S^{n} C \times S^m C  \to  S^m C$$
and the class we are interested in obtaining is $\ch(R^0p_*\mathcal{L}_{b,c})$. As stated in the outline of the paper, we apply the Grothendieck–Riemann–Roch theorem to the proper map $\pi$ and the line bundle $\mathcal{L}_{b,c}$. We get:

\eqtag{\ch\left(\sum_i (-1)^i R^i \pi_\star \mathcal{L}_{b,c}\right)=\pi_\star \left(e^{c_1(\mathcal{L}_{b,c})} \td\left(S^{n}C\right)\right).}{eq:GRR}

\begin{theorem}
    For any $w \in S^m C$, and for any $i>0$, $H^i(S^n C,\mathcal{L}_{b,c,w})=0$.
\end{theorem}

\begin{proof}
Fix $w \in S^m C$, then $$c_1(\mathcal{L}_{b,c,w})=c_1(\mathcal{L}_{b,c}|_{\lbrace w\rbrace})=c_1(\mathcal{L}_{b,c})|_{{\lbrace w\rbrace}}=b\theta_n+p\xi_n.$$
Let $\omega_{S^n C}$ be the canonical line bundle of $S^n C$. Then $$c_1(\mathcal{L}_{b,c,w} \otimes \omega_{S^n C}^{-1})=(b-1)\theta_n+(p+(n+1-g))\xi_n.$$ which has been shown to be ample when $p>-n-1+g$ and $b\geqslant 1$ in \cite{Klevtsov_Zvonkine_2025}.

\begin{corollary}
    By Grauert theorem \cite[Corollary 12.9]{Hartshorne_1977}, all higher direct images vanish and $$V_{b,c,d,g,n,m}:=\sum_i (-1)^iR^i \pi_\star \mathcal{L}_{b,c} = R^0 \pi_\star \mathcal{L}_{b,c} $$ is a vector bundle.
\end{corollary}

\end{proof}

What remains is to compute the pushforward $\pi_*$ in cohomology. We note that the map $\pi$ factors as follows:
\[
\begin{tikzcd}
S^n C \times S^m C \arrow[r, "\pi_1"] \arrow[dr, "\pi"'] & \operatorname{Pic}^n C \times S^m C \arrow[d, "\pi_2"] \\
& S^m C
\end{tikzcd}.
\]
Hence, our strategy will be to first compute the pushforward ${\pi_1}_*$, which corresponds to the integration over the fibers of $S^nC\to \pic^n C$, and then the pushforward ${\pi_2}_*$, which corresponds to the integration over $\pic^n C$. 

In order to carry out the first integration, we will need the following lemma. 

\begin{lemma}
\label{lem:firstpushforward_lemma1}
Let $n\geq 2g-1$, and $p$ be an integer. Let $\pi_1':S^nC \to \pic^n C$.
$${\pi_1}'_\star \left( e^{p \xi_n}\td(\xi_n)^{n+1-g}  e^{\theta_n \frac{\td \xi_n -\xi_n-1}{\xi_n}} \right)=f(\theta_n),$$
where $f$ is the polynomial defined as $$f(x)= \sum_{a\geqslant 0} \frac{1}{a!} \binom{n-g+p}{p-a} x^a.$$ We recall the convention that $\binom{n}{p}$ is zero whenever $p<0$ or $p>n$.

We refer to \cite[Lemma 4.2.2]{Aldonza_Dupont_2025}, \cite[Lemma 4.10]{Klevtsov_Zvonkine_2025} for the complete computation. This result follows from the fact that the pushforward by $\pi_1$ corresponds to integration along the projective fibers over $\mathrm{Pic}^n C$. The class $\theta_n$ is unchanged, while
\[
{\pi_1}_*\,\xi_n^{n-k}=\frac{\theta_n^k}{k!}.
\]

This formula, called the Poincaré formula (see \cite[\S 5]{Griffiths_book}) is obtained by noticing that fiber integration of the hyperplane class precisely selects the coefficient of $\xi_n^{n-g}$, as well as by using the relation
\[
\xi_n^{n-g+1}-\theta_n \,\xi_n^{n-g}+\cdots+(-1)^g\frac{\theta_n^g}{g!}\,\xi_n^{n-2g+1}=0.
\]

\end{lemma}

\begin{remark}\label{rem2}
Note that in the particular case $p=0$ we get ${\pi_1}'_\star \left(\td(S^nC)\right)=1$. This can also be seen as a consequence of the Grothendieck Riemann-Roch theorem. Applying the Grothendieck Riemann-Roch theorem to $\pi_1'$ and to the sheaf $\mathcal{O}{S^n C}$ gives
$$\ch\left(\sum_i R^i {\pi_1}'_\star \mathcal{O}_{S^{n}C}\right) \td\left(\pic^{n} C\right)={\pi_1}'_\star \left( \ch\left(\mathcal{O}_{S^{n}C}\right)\td\left( S^{n} C \right)\right)$$

But $\ch\left( \sum_i R^i {\pi_1}'_\star \mathcal{O}_{S^n C} \right) = \ch\left(R^0 {\pi_1}'_\star \mathcal{O}_{S^n C}\right)=\ch\left(\mathcal{O}_{\pic^{n} C}\right)=1$ and $\td\left(\pic^{n} C\right)=1$ as $\pic^{n} C$ is a complex torus. All together, it implies 

$${\pi_1}'_\star \left(\td\left(S^n C \right)\right)=1.$$

\end{remark}

\begin{remark}
Since $\pi_1$ does not act on the factor $S^m C$, the formula of lemma \ref{lem:firstpushforward_lemma1} is also true for $\pi_1:S^nC \times S^m C \to \pic^n C \times S^m C$ (up to pullbacks via projection maps):
$${\pi_1}_\star \left( \text{pr}_1^\star \left(e^{p \xi_n}\td(\xi_n)^{n+1-g}  e^{\theta_n \frac{\td \xi_n -\xi_n-1}{\xi_n}} \right)\right)={\text{pr}_1'}^\star f(\theta)$$
\end{remark}

\begin{proof}
This is a consequence of the base change formula in the Cartesian square 
\[
\begin{tikzcd}
S^n C \times S^m C \arrow[r, "\pi_1"] \arrow[d, "\mathrm{pr}_1"'] & \operatorname{Pic}^n C \times S^m C \arrow[d, "\mathrm{pr}'_1"] \\
S^n C \arrow[r, "\pi_1'"] & \operatorname{Pic}^n C
\end{tikzcd}
\]
which gives ${\pi_1}_\star \left( \text{pr}_1^\star \left(e^{p \xi_n}\td(\xi_n)^{n+1-g}  e^{\theta_n \frac{\td \xi_n -\xi_n-1}{\xi_n}} \right)\right)={\text{pr}_1'}^\star {\pi_1'}_\star \left(e^{p \xi_n}\td(\xi_n)^{n+1-g}  e^{\theta \frac{\td \xi_n -\xi_n-1}{\xi_n}} \right)$

\end{proof}

In order to compute the pushforward by $\pi_2$, we will first express the classes on $\pic^{n} C \times \pic^{m}C$ that have appeared so far in terms of a basis of degree $1$ forms on $\pic^n(C)$ and $\pic^m(C)$.

Let $a^1,\dots,a^g, b^1,\cdots,b^g$ be a basis of canonical cycles of $H_1(C,\mathbb{Z})$ and denote by $\alpha^1,\cdots,\alpha^g$, $\beta^1, \cdots,\beta^g$ the dual basis of real harmonic 1-forms.  Our choice of $z_0$ gives an Abel-Jacobi map $C \longrightarrow \pic^{n}(C)$. We denote by $\lbrace \alpha_n^r,\beta_{n}^r,  1\leq r\leq g \rbrace$ the symplectic basis of $H^1(\pic^{n}(C),\mathbb{Z})$ that pulls back to $\{\alpha^r,\beta^r, 1\leq r \leq g\}$ by the corresponding Abel-Jacobi map. 

We denote by $\lbrace \alpha_m^r,\beta_{m}^r,  1\leq r\leq g \rbrace$ the basis obtained by the same construction for $S^m C$.

\begin{lemma}
\label{lem:etainsimplecticbasis}
In terms of these symplectic bases, we have:
    
\eqtag{
\theta_n=\sum_{r=1}^g \alpha_n^r \wedge \beta_n^r, \qquad \theta_m=\sum_{r=1}^g \alpha_m^r\wedge \beta_m^r}{eq:ThetaAsDifferentialForm}
and
\begin{equation}
    \eta_{n,m}=\sum_{r=1}^g (\alpha_n^r \wedge\beta_m^r +\alpha_m^r \wedge\beta_n^r)
\end{equation}

\end{lemma}

\begin{proof}
The expression of $\theta_n$ and $\theta_m$ follows from the fact that, in cohomology, the polarization given by the theta divisor corresponds to the intersection pairing through the identification $H^1(J(C),\mathbb{Z})\simeq H^1(C,\mathbb{Z})$. Since $\lbrace \alpha_n^r+\alpha_m^r, \beta_n^r+\beta_m^r, 1\leq r\leq g \rbrace$ forms a symplectic basis of $\pic^{n+m}C$, we have with $\sigma$ the addition map:
    
\eq{
\eta_{n,m}&=\sigma^\star \theta_{n+m}-\theta_n-\theta_m\\
    &=\sum_{r=1}^g(\alpha_n^r+\alpha_m^r)\wedge(\beta_n^r+\beta_m^r)-\sum_{r=1}^g \alpha_n^r \wedge \beta_n^r  - \sum_{r=1}^g \alpha_m^r\wedge \beta_m^r\\
    &=\sum_{r=1}^g (\alpha_n^r \wedge\beta_m^r +\alpha_m^r \wedge\beta_n^r) }
    
    \end{proof}

\mainqhff*

 \begin{proof}
Since $p=d-bn-cm-b(g-1)=0$, $\ch(\mathcal{L}_{b,c})=e^{c_1(\mathcal{L}_{b,c})}$ is a pullback from $\pic^n C$ so we have 

\begin{align*}
    {\pi_1}_\star \left( e^{c_1(\mathcal{L}_{b,c})}\td(S^n C)\right)&=e^{c_1(\mathcal{L}_{b,c})} {\pi_1}_\star \left( \td(S^n C)\right)\\
    &=e^{c_1(\mathcal{L}_{b,c})}\\
    &=e^{b\theta_n + c \eta_{n,m}-cn\xi_m}
\end{align*}
where we used ${\pi_1}_\star \left( \td(S^n C)\right)=1$ by Remark \ref{rem2}, as well as the expression of $c_1(\mathcal{L}_{b,c})$ of proposition \ref{prop:firstcherclassunivquasiholebundle}. In terms of the symplectic basis $\alpha_n^i$ and $\beta_n^i$ (resp. $\alpha_m^i$ and $\beta_m^i$), we have:

\eq{
e^{b\theta_n + c \eta_{n,m}-cn\xi_m}&=e^{b\sum_{l=1}^g  \alpha_n^l \beta_n^l + c \sum_{l=1}^g \alpha_n^l \beta_m^l + c \sum_{l=1}^g \alpha_m^l \beta_n^l} e^{-cn\xi_m}\\
&=\prod_{l=1}^g e^{b \alpha_n^l \beta_n^l + c \alpha_n^l \beta_m^l + c \alpha_m^l \beta_n^l} e^{-cn\xi_m}
}

The second pushforward ${\pi_2}_\star$ corresponds to the integration along the fibers of the map $\pi_2$, that is to integration on $\pic^n C$. It selects only top terms containing $\prod_{l=1}^g \alpha_n^l \beta_n^l$, which is a volume form on $\pic^n C$, and outputs the coefficient in front of it. It does not act on $\xi_m$. Note that we can only obtain $\prod_{l=1}^g \alpha_n^l \beta_n^l$ if we pick the terms containing $\alpha_n^l \beta_n^l$ in each factor $e^{b\alpha_n^l \beta_n^l + c \alpha_n^l \beta_m^l + c \alpha_m^l \beta_n^l}$. In each of these factors, the coefficient in front of $\alpha_n^l \beta_n^l$ is given by $b-c^2 \alpha_m^l \beta_m^l=be^{-\frac{c^2}{b}\alpha_m^l \beta_m^l}$.
We thus obtain
\begin{align}\nonumber
{\pi_2}_\star \left( e^{\sum_{l=1}^g  b \alpha_n^l \beta_n^l + c \sum_{l=1}^g \alpha_n^l \beta_m^l + c \sum_{l=1}^g  \alpha_m^l \beta_n^l} e^{-cn\xi_m} \right)&= e^{-cn\xi_m} \prod_{l=1}^g be^{-\frac{c^2}{b}\alpha_m^l \beta_m^l}\\
&=b^g e^{-\frac{c^2}{b} \theta_m - cn \xi_m}
\end{align}
\end{proof}
Next we consider the case of non-localized quasiholes $p>0$.

\subsubsection{The case with non-localized quasiholes}

\mainqh*

\begin{proof}
For convenience, for any integer $i$ we write $[i]=\lbrace{1,\dots,i}\rbrace$. Furthermore, for $I\subset [g]$ a set, we write $(\alpha_n\beta_n)^I$ (resp. $(\alpha_m\beta_m)^I$) the product $\prod_{l\in I} \alpha_n^l \beta_n^l$ (resp. $\prod_{l\in I} \alpha_m^l \beta_m^l$).

We have 
\begin{align*}
f(\theta_n)&=\sum_r \frac{1}{r!}\binom{n-g+p}{p-r} \left( \sum_{l=1}^g  \alpha_n^l \beta_n^l\right)^r\\
&= \sum_r \binom{n-g+p}{p-r} \sum_{I \subset [g],|I|=r} \left(\alpha_n \beta_n\right)^I
\end{align*}
Thus
\begin{align*}
{\pi_2}_\star \left( f \left(\theta_n \right)  e^{ b \theta_n + c \eta_{n,m}} \right)&=
{\pi_2}_\star \left( f \left(\theta_n\right)  e^{\sum_l (b \alpha_n^l \beta_n^l + c (\alpha_n^l \beta_m^l + \alpha_m^l \beta_n^l))} \right)\\
&= {\pi_2}_\star \left(e^{\sum_l (b \alpha_n^l \beta_n^l + c (\alpha_n^l \beta_m^l + \alpha_m^l \beta_n^l))} \sum_r \binom{n-g+p}{p-r} \sum_{I \subset [g],|I|=r} \left(\alpha_n \beta_n\right)^I\right)\\
&=  \sum_{I\subset [g]} \binom{n-g+p}{p-|I|} {\pi_2}_\star \left(e^{\sum_l (b \alpha_n^l \beta_n^l + c (\alpha_n^l \beta_m^l + \alpha_m^l \beta_n^l))} \left(\alpha_n \beta_n\right)^I\right)\\
\end{align*}
In order for a term from $e^{\sum_l (b \alpha_n^l \beta_n^l + c (\alpha_n^l \beta_m^l + \alpha_m^l \beta_n^l))}$ to contribute, it must complete the monomial $\left(\alpha_n \beta_n\right)^I=\prod_{l\in I} \alpha_n^l \beta_n^l$ to a full monomial $\prod_{l=1}^g \alpha_n^l \beta_n^l$. In 
$$e^{b \alpha_n^l \beta_n^l + c (\alpha_n^l \beta_m^l + \alpha_m^l \beta_n^l)}=1+c(\alpha_n^l \beta_m^l + \alpha_m^l \beta_n^l) + \alpha_n^l \beta_n^l (b-c^2 \alpha_m^l \beta_m^l)$$ we need to choose the term $\alpha_n^l \beta_n^l (b-c^2 \alpha_m^l \beta_m^l)$ for each $l$ not in $I$ and $1$ otherwise.
We thus get $${\pi_2}_\star \left(e^{\sum_l (b \alpha_n^l \beta_n^l + c (\alpha_n^l \beta_m^l + \alpha_m^l \beta_n^l))} \left(\alpha_n \beta_n\right)^I\right)=\prod_{l\in I^c} (b-c^2 \alpha_m^l \beta_m^l).$$
Developing this product and re-indexing the sums, we get 
\eq{
 \sum_{I\subset [g]} \binom{n-g+p}{p-|I|} \prod_{l\in I^c} (b-c^2 \alpha_m^l \beta_m^l) &=\sum_{I\subset [g]} \binom{n-g+p}{p-|I|} \sum_{F \subset I^c} (-c^2)^{|F|} (\alpha_m \beta_m)^F b^{|I^c|-|F|} \\
 &=\sum_{k=0}^g \binom{n-g+p}{p-g+k}\sum_{F\subset [g]} (-c^2)^{|F|} (\alpha_m \beta_m)^F b^{k-|F|} \binom{g-|F|}{k-|F|}\\
}
where the binomial coefficient $\binom{g-|F|}{k-|F|}$ appears because for each fixed $F$, there are $\binom{g-|F|}{k-|F|}$ sets $I \subset [g]$ of size $g-k$ such that $F\subset I^c$. 
Regrouping $(\alpha_m \beta_m)^F$ to form $\theta_m$ since $$j! \sum_{F \subset [g],|F|=j} (\alpha_m \beta_m)^F=  {\theta_m}^j,$$ we get: 

\begin{align}
    {\pi_2}_\star \left(e^{-cn \xi_m} f \left(\theta_n \right)  e^{ b \theta_n + c \eta_{n,m}} \right)&=e^{-cn \xi_m}\sum_k \sum_j \frac{1}{j!}(-c^2 \theta_m)^j b^{k-j} \binom{n-g+p}{k-g+p}\binom{g-j}{k-j}
\end{align}

\end{proof}

We are now going to carry out this computation using explicit wave-functions.

\section{Recovering the Chern classes from explicit wave-functions in low genus}
\label{sec:explicit}

In the maximally filled case,
\begin{equation}
d = bn + cm + b(g-1),
\end{equation}
the bases for the Laughlin states without quasiholes have been constructed explicitly for $g=0$ \cite{Haldane1983}, for $g=1$ \cite{Haldane1985} and for $g>1$ in \cite{Klevtsov2019}. The bases with quasiholes can be rather straightforwardly inferred from those papers. Here, we focus on the cases $g=0$ and $g=1$. In the latter case, Laughlin states with quasiholes were constructed in \cite{Einarsson1990}. The explicit bases that we use in what follows for the family of sections $\{s_{i}(.|w) \}_{1 \le i \le b^g}$ of the line bundle
$L_w := L_{b,c,w}$ for $w \in \mathbb{C}^m$ are of the type used in \cite{Klevtsov_2016,Burban_Klevtsov_2025,burban2024norms}. In genus $0$ the holomorphic sections are constructed in terms of polynomials in projective coordinate, and in genus $1$ using theta functions. In both cases, this representation amounts to fixing a trivialization of the line bundles. By abuse of notation, we will denote sections and their trivializations by the same symbols.
Here, the goal is to recover the Chern characters $$\ch_i(V),\, i>0$$ of the vector bundle $V := V_{b,c,d,g,n,m}$ from these families of sections by computing the curvature of the natural Berry-Chern connection on this bundle. For simplicity, throughout this discussion we restrict ourselves to the case of a particle--quasihole interaction $c=1$. Note that this method does not provide the rank $\rk V = \ch_0(V)$ of the bundle. 

Let us fix a Riemannian metric, and thus a volume form $d{\rm Vol }$ on the curve $C$. We will also need a hermitian metric $(,)_{L_w}$ on $L_w$ for each $w$. Locally, it can be written as 
$$(s(z|w),s'(z|w))_{L_w}=h_w(z,\bar z)s(z|w)\ov{ s'(z|w)},$$
where the two form $-i\partial_z\bar\partial_z\log h_w$ is a global positive definite form and $\int_C -i\partial_z\bar\partial_z\log h_w=2\pi d$ for all $w\in C^m$.

Then the hermitian scalar product on $V$ is defined as an $L^2$ product
$$\langle s(.|w),s'(.|w) \rangle_{V} = \int_{S^n C} (s(z|w),s'(z|w))_{L_w} d{\rm Vol }_n(z),$$
where $d{\rm Vol }_n(z)$ is the descent of $\prod_{1\leq \mu \leq n} d{\rm Vol }(z_\mu)$ from $C^n$ to $S^n C$.

The operator $\ov{\partial}$ on $S^nC\times S^mC$ splits as $\ov{\partial}_n \otimes \text{id} + \text{id} \otimes \ov{\partial_m}$, and this $\ov{\partial}_m$ gives a Dolbeault operator $\ov{\partial}: s \to \left(z\to \ov{\partial}_w s(z|w)\right)$ on $V$.

Now, for each $w \in S^m C$, let $\{z \to s_{i}(z|w),1\leq i \leq b^g\}$ be a basis of sections  of the bundle of Laughlin states $L_w$ above $S^n C$ and suppose that this frame is holomorphic in $w$, in the sense that $\ov{\partial}_ms=0$. We then get a local holomorphic frame of $V$ as $$\{ w\to s_i(.|w), 1\leq i \leq b^g\}.$$ 
Let $$H=(\langle s_i(.|w),s_{j}(.|w)\rangle_{V})_{1\leq i,j\leq b^g}.$$ 
The Chern connection associated to $\langle , \rangle_V$ is given by the one form $\omega=H^{-1} \partial_m H$ and its curvature reads
$$\Omega=\ov{\partial}_m(H^{-1} \partial_m H)=-\partial_m\ov{\partial}_m\log H$$
from which one can recover the Chern characters of $V$ as 
$$\ch(V)=\left[ \tr(e^{\frac{i}{2\pi}\Omega})\right].$$
We will also need the following technical lemma.
\begin{lemma}
\label{lem:descendofsumxi}
Let $f:C^m \to S^mC$ and $\pi_i: C^m \to C$ the $i$-th projection. Denote by $[z_0]$ the class of the point $z_0$ on $C$. Then $\sum_i \pi_i^\star [z_0]=f^\star \xi_m$. 

\end{lemma}

\begin{proof}
By Eq. \eqref{eq:xisecondresentation}, $\xi_m$ is the locus of points in $S^m C$ that contains at least once $z_0$. Set-theoretically, we have $f^{-1} \xi_m=\cup_i \pi^{-1}(\lbrace z_0\rbrace)$. Since there is no multiplicity, we have $f^\star \xi_m=\sum_i \pi_i^\star [z_0]$.

\end{proof}

\subsection{The case $C=\pone$}

Let $z$ be a projective coordinate on $\pone$ and let $h(z,\bar z)=\frac{1}{1+|z|^2}$

The class $-\frac{i}{2\pi}\partial \ov{\partial} \log h$ represents the class of a point on $\pone$. By lemma \ref{lem:descendofsumxi},  $-\frac{i}{2\pi}\partial_m \ov{\partial}_m \log \prod_{\gamma} h(w_\gamma)$ descends to the class $\xi_m$ on $\mathbb{P}^m$ through the map $f:(\pone)^{m}\to S^m \pone \simeq \mathbb{P}^m$.

By theorem \ref{thm:mainqhff}, the space of Laughlin states with quasiholes at position $w=w_1+\dots+w_m$ is one dimensional. It is spanned by $$s(z_1,\dots,z_n|w)=\prod_{1\leqslant\mu<\nu\leqslant n} (z_\mu - z_\nu)^b \prod_{\substack{1\leqslant\mu\leqslant n\\1\leqslant \gamma\leqslant m}}(z_\mu-w_\gamma).$$ 
For each $1\leq \mu \leq n$, $z_\mu \to s(z_1,\dots,z_n|w)$ has $b(n-1)+m=d$ zeroes in $z_\mu$. This implies that $$z_\mu\to h(z_\mu)^{b(n-1)+m} |s(z_1,\dots,z_n|w)|^2$$ is a globally defined positive definite function on $C$ and therefore that
$$(s(z|w),s'(z|w))_{L_w} = s(z|w)\ov{s'(z|w)} \prod_\mu h(z_\mu)^d$$ defines a hermitian metric on $L_w$.  The associated hermitian metric on $V$ is given by 
\begin{equation}
\label{eq:L2sphere}
\langle s(.|w),s(.|w) \rangle_V=\int_{S^n C}\prod_{\mu<\nu} |z_\mu - z_\nu |^{2b} \prod_{\mu,\gamma}  |z_\mu-w_\gamma|^2 \prod_\mu h(z_\mu)^d\; d{\rm Vol}_n(z).
\end{equation}

Note also that $w_\gamma\to s(z_1,\cdots,z_n|\omega)$ has $n$ zeros in each variable $w_\gamma$, thus $$s(z_1,\dots,z_n|w)\ov{s'(z_1,\dots,z_n|w)}\prod_\mu h(z_\mu)^d \prod_\gamma h(w_\gamma)^n$$ is a globally defined function on $S^n C \times S^m C$. This implies that $\langle s(.|w),s(.|w) \rangle_V \prod_{\gamma} h(w_\gamma)^{n}$ is a strictly positive globally defined function on $S^m C$. It follows that in cohomology, $$\left[\partial_m \ov{\partial}_m \log \left(\langle s(.|w),s(.|w)\rangle_V \prod_\gamma h(w_\gamma)^n\right)\right]=0.$$

Let $H=\langle s(.|w),s(.|w) \rangle_V $. We have 

\begin{align}\nonumber
    c_1(V)&= \left[\frac{i}{2\pi}\ov{\partial }_m (H^{-1} \partial_m H)\right]\\\nonumber
    &=\left[-\frac{i}{2\pi}\partial_m \ov{\partial}_m\log H\right]\\\nonumber
    &=\left[-\frac{i}{2\pi}\partial_m \ov{\partial}_m \log \prod_\gamma h(w_\gamma)^{-n}\right]\\
    &=-n\xi_m.
\end{align}

Since the rank of $V$ is one, we have 
$$\ch(V)=1+c_1(V)=e^{-n \xi_m}.$$

\subsection{The genus $1$ case}

Let $\tau \in \mathbb{C}$ with $\Im(\tau)>0$ and let $C=\mathbb{C}/({\mathbb{Z}+\tau \mathbb{Z}})$ be an elliptic curve. On the complex plane, we write $h:(z,\ov{z}) \mapsto \exp(-\frac{2\pi}{\Im(\tau)} \Im(z)^2)$ for a Hermitian metric on the line bundle of degree one. The associated $(1,1)$ form $-\frac{i}{2\pi}\partial \ov{\partial} \log h$ corresponds to the first Chern class and its Poincaré dual is the class of a point on $C$. 

We have
\begin{lemma}
\label{lem:repOfThetaAndXiInGenusOne}
The two form on $C^m$
$$-\frac{i}{2\pi}\partial_m \ov{\partial}_m \log e^{-\frac{2\pi}{\Im(\tau)}\Im(\sum_\gamma w_\gamma)^2}=\frac{i}{2\Im(\tau)}d\big(\sum_{\gamma=1}^m w_\gamma\big)\wedge d\big(\sum_{\gamma=1}^m\bar w_\gamma\big)$$ descends to the class $\theta_m$ on $S^mC$ and $$-\frac{i}{2\pi}\partial_m \ov{\partial}_m \log e^{-\frac{2\pi}{\Im(\tau)}\sum_\gamma \Im(w_\gamma)^2}=\frac{i}{2\Im(\tau)}\sum_{\gamma=1}^mdw_\gamma \wedge d\bar w_\gamma$$ descends to the class $\xi_m$ on $S^mC$.
\end{lemma}

\begin{proof}
We start with the first claim.  $$\{ w_1\dots,w_m\} \mapsto e^{-\frac{2\pi}{\Im(\tau)}\Im(\sum_\gamma w_\gamma)^2}$$ descends on $S^mC$, but since it depends only on $\sum_\gamma w_\gamma$, it in fact further descends to $\pic^m{C}$. We can assume $C$ to be generic, in which case the Neron-Severi group of $\pic^m C$ is generated by $\theta_m$. $e^{-\frac{2\pi}{\Im(\tau)}\Im(\sum_\gamma w_\gamma)^2}$ thus descends to a class $c\theta_m$ with $c\in \mathbb{Z}$.
As we are in the genus one case, the $\theta$ divisor corresponds to the class of a point, that of the origin $O$ of the elliptic curve (recall this comes from the description of $\theta$ as $\lbrace L \in \pic^{g-1}, h^0(C,L)>0\rbrace$).
    
Furthermore, the divisor $\sum_\gamma w_\gamma-m O$ corresponds to the point $z=\sum_\gamma w_\gamma$ on the elliptic curve through $C \simeq \pic^m C$. As $$e^{-\frac{2\pi}{\Im(\tau)}\Im(\sum_\gamma w_\gamma)^2}=e^{-\frac{2\pi}{\Im(\tau)}\Im(z)^2}$$ corresponds to the class of a point on the elliptic curve, we obtain $c=1$. 
For the second claim, note that $-\frac{i}{2\pi}\partial_m \ov{\partial}_m \log e^{-\frac{2\pi}{\Im(\tau)}\sum_\gamma \Im(w_\gamma)^2 }$ is the sum of pullbacks of the class of a point to $C^m$. the result follows by lemma \ref{lem:descendofsumxi}.
\end{proof}

Holomorphic sections of line bundle on the torus can be constructed with the help of theta functions, see e.g. \cite{mumford1983tata}. 
\begin{definition}[Theta functions]
Let $\tau \in \mathbb{C}$ with $\Im(\tau)>0$. We define the theta functions with parameter $\tau$ as:
$$z \to \theta(z,\tau)=\sum_{n\in\mathbb{Z}} \exp(\pi i n^2 \tau + 2 \pi i n z).$$

\end{definition}

This function is quasi-periodic and satisfies:

$$\theta(z+1,\tau)=\theta(z,\tau)$$
$$\theta(z+\tau,\tau)=\exp(-i\pi\tau -2i\pi z)\theta(z,\tau)$$

We also define the so-called theta function with characteristic $a,b\in\mathbb R$ as
\eq{
\theta \chara{a}{b}(z,\tau)&= \sum_{n\in\mathbb{Z}} \exp(\pi i(n + a)^2\tau + 2\pi i(n + a)(z + b))\\
&=\exp \left( \pi i a^2 \tau + 2\pi i a(z+b) \right) \theta(z+a\tau+b,\tau)
}
which satisfies
\begin{equation}
    \label{eq:thetashift}
\theta \chara{a}{b}(z+n+m\tau,\tau)=\exp(-i\pi m^2 \tau - 2\pi i m z +2\pi i (an+bm)) \theta \chara{a}{b}(z,\tau)
\end{equation}
 for any $n,m \in \mathbb{Z}$. In particular, the first Jacobi theta function
 $$
\theta_1(z,\tau):=\theta \chara{1/2}{1/2}(z,\tau)
$$ 
has exactly one simple zero at $z=0$. We also note that
\begin{proposition}
It follows immediately from the periodicity properties above that
\label{prop:theta}
$$\left|\theta \chara{a}{b}(z,\tau)\right|^2\exp\left(-\frac{2\pi}{\Im(\tau)} \Im(z)^2\right),$$
is a globally defined real-valued function on the torus.
\end{proposition}

Laughlin states with quasiholes on the torus were first constructed in \cite{Einarsson1990}, following \cite{Haldane1985}. 
We refer to \cite{Klevtsov_2016,Burban_Klevtsov_2025,burban2024norms} for the following basis and for the proofs of its properties.

\begin{proposition}
\label{prop:basisgenusone}
The following is a basis of Laughlin states with $m$ localized quasiholes at positions $\{w_1,\dots,w_m\}$,
$$s_{l}(z_1,\dots,z_n|w)=\theta \chara{\frac{l}{b}}{0} \left(b\sum_{\mu=1}^n z_\mu + \sum_{\gamma=1}^m w_\gamma,b\tau\right) \prod_{ 1\leqslant\mu<\nu\leqslant n} \theta_1(z_\mu - z_\nu,\tau)^b \prod_{\substack{1\leqslant\mu\leqslant n\\1\leqslant\gamma\leqslant m}}\theta_1(z_\mu-w_\gamma,\tau)$$ for $0\leq l \leq b-1$.

\end{proposition}

\begin{proof}
Those quasi-periodic functions correspond to the sections of $L^{\boxtimes n}$ which correspond to the Laughlin states: they have the correct vanishing order, are (anti-)symmetric in each variable for $b$ even (odd). Furthermore, their restrictions to a single copy of $C$ give a section of fixed line bundle of degree $d=bn+cm$, as follows from counting the zeros in each $z_\mu$. 
Showing that the family is free is equivalent to showing that the family 
\begin{equation}
    \label{eq:stilde}
z \mapsto \tilde{s}_{l}(z|w)=\theta \chara{\frac{l}{b}}{0} \left(\sum_\mu b z_\mu + \sum_\gamma w_\gamma,b\tau\right)
\end{equation}
is free, which is a classical result, see e.g. \cite[Ch.2, Prop.1.3]{mumford1983tata}. 

This coincides with the topological degeneracy $b^g=b$ for the space of Laughlin states with localized quasiholes, which follows from the Theorem \ref{thm:mainqhff}.
\end{proof}

\paragraph{The determinantal case.}
We first treat the case $b=1$. By proposition \ref{prop:basisgenusone}, for a given $w=w_1+\dots+w_m\in S^m C$ the space of states with quasiholes is one dimensional and spanned by $$s\left(z_1,\dots,z_n|w\right)=\theta\bigg(\sum_\mu z_\mu + \sum_\gamma w_\gamma,\tau\bigg) \prod_{\mu<\nu} \theta_1(z_\mu - z_\nu,\tau) \prod_{\mu,\gamma}\theta_1(z_\mu-w_\gamma,\tau).$$

The following is a well-defined hermitian metric on $L_{w}$ above $S^n C$ 
$$(s(z|w),s'(z|w))_{L_w} = s(z|w) \ov{s'(z|w)} e^{-\frac{2\pi}{\Im(\tau)}(n+m)\sum_\mu \Im(z_\mu)^2}.$$
This is a globally defined function on $S^nC$. The associated Hermitian metric on $V$ is given by:
$$H:=\langle s(.|w),s'(.|w) \rangle_V = \int_{S^n C} s(z|w) \ov{s'(z|w)} e^{-\frac{2\pi}{\Im(\tau)}(n+m)\sum_\mu \Im(z_\mu)^2} d{\rm Vol}_n(z).$$

On the other hand 
$$s(z|w) \ov{s'(z|w)} e^{-\frac{2\pi}{\Im(\tau)} \left( \Im(\sum_\mu z_\mu+\sum_\gamma w_\gamma)^2+\sum_{\mu<\nu}\Im(z_\mu-z_\nu)^2 +\sum_{\mu,\gamma} \Im(z_\mu-w_\gamma)^2\right)}$$
is a globally defined function over $S^n C\times S^mC$, as follows from Prop. \ref{prop:theta}. Using the relations
\begin{align*}
(n+m)\left( \sum_\mu \Im(z_\mu)^2 + \sum_\gamma \Im(w_\gamma)^2 \right)&=\Im(\sum_\mu z_\mu)^2+\Im(\sum_\gamma w_\gamma)^2+ \sum_{\mu< \nu} \Im(z_\mu-z_\nu)^2\\
&+\sum_{\gamma< \delta} \Im(w_\gamma-w_\delta)^2+\sum_{\mu,\gamma} \Im(z_\mu-w_\gamma)^2,\\
m\sum_\gamma\Im(w_\gamma)^2&=\Im(\sum_\gamma w_\gamma)^2+ \sum_{\gamma< \delta} \Im(w_\gamma-w_\delta)^2
\end{align*}
we obtain
\begin{align}\label{exprel}\nonumber
   &e^{-\frac{2\pi}{\Im(\tau)}\Im (\sum_\mu z_\mu + \sum_\gamma w_\gamma)^2} e^{-\frac{2\pi}{\Im(\tau)} \sum_{\mu< \nu} \Im(z_\mu-z_\nu)^2} e^{-\frac{2\pi}{\Im(\tau)} \sum_{\mu,\gamma} \Im(z_\mu - w_\gamma)^2}\\
   &=e^{-\frac{2\pi}{\Im(\tau)}\left((n+m)\sum_\mu \Im(z_\mu)^2-n\sum_\gamma \Im(w_\gamma)^2-\Im(\sum_\gamma w_\gamma)^2\right)}
\end{align} 
It follows that $He^{-\frac{2\pi}{\Im(\tau)} \left( n \sum_\gamma \Im(w_\gamma)^2  + \Im(\sum_\gamma w_\gamma)^2\right)}$ is a strictly positive function on $S^mC$. Applying $\partial \ov{\partial}\log(.)$ to this function gives an exact form. Hence in cohomology,
\begin{align*}
    c_1(V)&=\left[-\frac{i}{2\pi}\partial_m \ov{\partial}_m\log H\right]\\
    &=\left[-\frac{i}{2\pi}\partial_m \ov{\partial}_m \log e^{\frac{2\pi}{\Im(\tau)} \left( n\sum_\gamma \Im(w_\gamma)^2  + \Im(\sum_\gamma w_\gamma)^2\right)}\right]\\
    &=-\theta_m-n\xi_m,
\end{align*}
by Lemma \ref{lem:repOfThetaAndXiInGenusOne}. Since $V$ has rank one, we have $$\ch(V)=1+c_1(V)=e^{-\theta_m-n \xi_m}.$$

\paragraph{The fractional case.}

In the fractional case, the following is a well-defined hermitian metric on $L_{w}$ above $S^n C$:

$$(s(z|w),s'(z|w) )_{L_w} = s(z|w) \ov{s'(z|w)} e^{-\frac{2\pi}{\Im(\tau)}(bn+m)\sum_\mu \Im(z_\mu)^2}.$$ The associated hermitian metric on $V$ is given by:
\begin{equation}
\label{eq:L2torus}
\langle s(.|w),s'(.|w) \rangle_V = \int_{S^n C} s(z|w) \ov{s'(z|w)} e^{-\frac{2\pi}{\Im(\tau)}(bn+m)\sum_\mu \Im(z_\mu)^2} d{\rm Vol}_n(z).
\end{equation}
A calculation completely analogous to Eq.\ \eqref{exprel} leads to the following relation
\begin{align*}
e^{-\frac{2\pi}{\Im(b\tau)} \left(\Im(b\sum_\mu z_\mu+\sum_\gamma w_\gamma)^2+b\sum_{\mu<\nu}\Im(z_\mu-z_\nu)^2 +\sum_{\mu,\gamma} \Im(z_\mu-w_\gamma)^2\right)}\\=e^{-\frac{2\pi}{\Im(\tau)} \left((bn+m) \sum_\mu \Im(z_\mu)^2 + \frac{1}{b} \Im(\sum_\gamma w_\gamma)^2 + n\sum_\gamma \Im(w_\gamma)^2 \right)}.
\end{align*}

Let $H=(\langle s_i(.|w),s_j(.|w)\rangle)_{V,1\leq i,j\leq b}$ be the Gram matrix associated to the basis of Prop.\ \ref{prop:basisgenusone} and the inner product Eq.\ \eqref{eq:L2torus}. We want to compute

\begin{align}
    \ch(V)=\left[\tr  \left( e^{- \frac{i}{2\pi}\partial_m \ov{\partial}_m \log H}\right)\right],
\end{align} where we used the fact that for each $w \in S^m C$, $H$ has non zero eigenvalues as the Gram matrix of a basis.

Let $$h(w)=e^{-\frac{2\pi}{\Im(\tau)} \left(\frac{1}{b} \Im(\sum_\gamma w_\gamma)^2 + n\sum_\gamma \Im(w_\gamma)^2 \right)}.$$
The difficulty is that in this case $H\cdot h(w)$ is not an invariant function on $S^mC$. However, this will not prevent us from carrying out the computation of the Chern character due to the following two lemmas. 

\begin{lemma}
Under $w_\gamma \to w_\gamma+1$, we have 
\begin{equation}
   \langle s_l(.|w),s_k(.|w) \rangle_V h(w) \mapsto q^{l-k}\langle s_{l}(.|w),s_{k}(.|w) \rangle_V h(w)
  \end{equation}
  where $q=e^{\frac{2\pi i}{b}}$.
Under $w_\gamma \to w_\gamma+\tau$, we have 
 
\begin{equation}
  \langle s_i(.|w),s_j(.|w) \rangle_V h(w) \mapsto \langle s_{i+1}(.|w),s_{j+1}(.|w) \rangle_V h(w) 
   \end{equation} with the convention $s_b=s_0$.
\end{lemma}

\begin{proof}
For the shift $w_\gamma \mapsto w_\gamma+1$, the formula above follows directly from Eq. \eqref{eq:thetashift} and Prop. \ref{prop:theta}. 

For the shift $w_\gamma \mapsto w_\gamma+\tau$  the following transformation  formula holds
\begin{equation}\label{trans}
\tilde{s}_{l}(z_1,\dots,z_n|w) \mapsto e^{-\pi i \frac{\tau}{b}} e^{-2\pi i \frac{1}{b}\left(b\sum_\mu z_\mu + \sum_\gamma w_\gamma\right)} \tilde{s}_{l+1}(z_1,\dots,z_n|w),
\end{equation}
where
$$z \mapsto \tilde{s}_{l}(z|w)=\theta \chara{\frac{l}{b}}{0} \left(\sum_\mu b z_\mu + \sum_\gamma w_\gamma,b\tau\right)$$
was introduced in Prop. \ref{prop:basisgenusone} and Eq. \eqref{eq:stilde}. 
The proof of Eq.\ \eqref{trans} is analogous to \cite[Lem. 2.9]{Burban_Klevtsov_2025} and we leave it to the reader.
\end{proof}

\begin{lemma}
  Let $H'=Hh$. In cohomology $[\tr e^{-\frac{i}{2\pi}\partial_m \ov{\partial}_m \log H'}]=\rk(V)$.
\end{lemma}

\begin{proof}
The first term in the Taylor expansion of the exponential gives $\rk(V)$ when taking the trace, and we have to show that $\forall k>0$, $[\tr \left( (\partial_m \ov{\partial}_m \log H')^k\right)]=0$.

Set $$\chi=\tr \left( \left(\partial_m \ov{\partial}_m \log H'\right)^{k-1} \ov{\partial}_m \log H' \right).$$ 
From the previous Lemma it follows that
under $w_\gamma \to w_\gamma+1$ we have $H' \mapsto PH'P^{-1}$ and under $w_\gamma \to w_\gamma+\tau$, $H' \mapsto QH'Q^{-1}$, where $P$ and $Q$ are the following constant matrices
\[
P=
\begin{pmatrix}
1 & 0 & 0 & \cdots & 0 \\
0 & q & 0 & \cdots & 0 \\
0 & 0 & q^2 & \cdots & 0 \\
\vdots & \vdots & \vdots & \ddots & \vdots \\
0 & 0 & 0 & \cdots & q^{b-1}
\end{pmatrix},\quad\quad\quad Q= 
\begin{pmatrix}
0 & 0 & \cdots & 0 & 1 \\
1 & 0 & \cdots & 0 & 0 \\
0 & 1 & \cdots & 0 & 0 \\
\vdots & \vdots & \ddots & \vdots & \vdots \\
0 & 0 & \cdots & 1 & 0
\end{pmatrix},
\]
where recall $q=e^{\frac{2\pi i}{b}}$.
Upon conjugation of $H'$ by any invertible constant matrix $U$, $\chi$ is invariant. Indeed,
\begin{align*}
    &\tr \left( \left(\partial_m \ov{\partial}_m (\log (UH'U^{-1})\right))^{k-1} \ov{\partial}_m (\log (UH'U^{-1}))\right)\\
    &=\tr \left( \left(\partial_m \ov{\partial}_m (U\log H'U^{-1})\right)^{k-1} \ov{\partial}_m (U\log H'U^{-1})\right)\\
    &=\tr \left(  \left(U(\partial_m \ov{\partial}_m \log H' )U^{-1}\right)^{k-1} U\ov{\partial}_m (\log H')U^{-1}\right)\\
    &=\tr \left( U \left(\partial_m \ov{\partial}_m \log H' \right)^{k-1} \ov{\partial}_m \log H' U^{-1}\right)\\
    &=\chi
\end{align*} and thus $\chi$ is a smooth form defined on $S^m C$. Furthermore, $d\chi=\tr \left(\partial_m \ov{\partial}_m \log H'\right)^k$ hence the result.

\end{proof}

We can now recover explicitly the result of theorem \ref{thm:mainqhff} and get the Chern character of $V$. We have $$\ch(V)=[\tr e^{\frac{i}{2\pi} \Omega}]$$
and we obtain the following proposition.

\propclasschernconnection*

\begin{proof}
\begin{align*}
    \left[ \tr e^{\frac{i}{2\pi} \Omega} \right]&=\left[ \tr \left( e^{-\frac{i}{2\pi}\partial_m \ov{\partial}_m \log Hh} e^{-\frac{i}{2\pi}\partial_m \ov{\partial}_m \log h^{-1}}\right) \right]\\
    &=\left[ \tr \left( e^{-\frac{i}{2\pi}\partial_m \ov{\partial}_m \log H'} \right) e^{\frac{i}{2\pi}\partial_m \ov{\partial}_m \log h}\right]\\
    &=\rk(V) e^{\left[\frac{i}{2\pi} \partial_m \ov{\partial}_m \log h\right] }
 \end{align*}   
    
    By Lemma \ref{lem:repOfThetaAndXiInGenusOne}, $$-\bigg[\frac{i}{2\pi}\partial_m \ov{\partial}_m \log h\bigg]=\frac{1}{b}\theta_m+n\xi_m$$ and thus 
    $$\ch(V)=be^{-\frac{1}{b}\theta_m-n\xi_m}$$
\end{proof}

\section{Charge transport with quasiholes, and quasihole statistics}

 It is interesting to consider the variations of the holomorphic line bundle $L$ representing the magnetic field. The moduli space of holomorphic line bundles $L$ is the Picard variety $\pic^d(C)$. Following \cite{Avron_Seiler_Zograf_1994}, the adiabatic transport on $\pic^d(C)$ corresponds to varying the Aharonov-Bohm fluxes through the holes of the surface and thus induces the electric Hall current in the sample. The first Chern class of the bundle of quantum Hall states corresponds to the quantized Hall conductance. For Laughlin states, the Hall conductance equals the slope of the Laughlin bundle \cite{Klein_Seiler_1990}, i.e. the ratio of the first Chern class and the rank. It was computed in Ref.\ \cite{Klevtsov_Zvonkine_2025} for Laughlin states and for multilayer states in Ref.\ \cite{Aldonza_Dupont_2025}. 

In the paradigm of the Laughlin states, Eq.\ \eqref{Psib}, if the particles $z_\mu$ are assigned a unit electric charge, the quasihole has a fractional charge $c/b$ of the opposite sign. This is because setting $c=b$ and replacing $w_1$ by $z_{N+1}$ in Eq.\ \eqref{Psib} is equivalent to an insertion of $N+1$ particles. Thus, the quasiholes $w_\gamma$ are electrically charged and participate in the Hall transport on the same footing as particles $z_\mu$. The goal of this section is to capture their Hall conductance via the Chern character calculation. 

With this goal in mind, we construct a universal line bundle $\mathcal{U}_{b,c}$ above $S^n C \times S^m C \times \pic^d C$ that we then push forward on $S^m C \times \pic^d C$.

\begin{proposition}
There exists a unique line bundle $\mathcal{U}_{b,c}$ above $S^n C \times S^m C \times \pic^d$ such that
$${\mathcal{U}_{b,c}}_{|S^n C \times S^m C \times \{[L]\} } \simeq \mathcal{L}_{b,c}$$
$${\mathcal{U}_{b,c}}_{ \{(n z_0,m z_0)\} \times \pic^d C} \text{ is trivial.}$$

\end{proposition}

\begin{proof}
Uniqueness is given by the SeeSaw theorem. For existence, let $P\to C \times \pic^d C$ be the poincaré line bundle such that 

$$P_{|C\times \{[L]\}}\simeq L$$
$$P_{|\{z_0\}\times \pic^d C} \text{ is trivial.}$$

One can then consider $P^{\boxtimes n}$ on $C^n\times \pic^d C$ and pull it back to $C^n \times C^m \times \pic^d C$. Tensoring by the pullback of the divisor $\Delta_{n,m}$ on $C^n \times C^m$ to $C^n \times C^m \times \pic^d C$, we get a line bundle which descends to $\mathcal{U}_{b,c}$ a line bundle which satisfies the proposition.

\end{proof}

By analogy with \ref{lem:etainsimplecticbasis}, we introduce the cohomology classes
$$\eta_{n,d}=\sum_{r=1}^g (\alpha_n^r \wedge\beta_d^r +\alpha_d^r \wedge\beta_n^r), \, \eta_{m,d}=\sum_{r=1}^g (\alpha_m^r \wedge\beta_d^r +\alpha_d^r \wedge\beta_m^r)$$
where $\alpha_d^l,\beta_d^l$ is a symplectic basis of the degree one cohomology of $\pic^d C$.

\begin{proposition}
    $$c_1(\mathcal{U}_{b,c})=b\theta_n + c \eta_{n,m} + p \xi_n -n c \xi_m - \eta_{n,d}$$
    where $$d-p=b(n+g-1)+cm$$ and $\eta_{n,d}$ the mixed class pulled back from $\pic^nC \times \pic^d C$. 
\end{proposition}

\begin{proof}
The proof is similar to that of Prop. \ref{prop:firstcherclassunivquasiholebundle}. The only difference is that we start from the descent to $S^n C \times \pic^{d}$ of $P^{\boxtimes n}$ whose first Chern class is $d\xi_n +\eta_{n,d}$ instead of the descent of $L^{\boxtimes n}$ to $S^n C$ whose first Chern class was $d \xi_n$.

We get that the difference between $c_1(\mathcal{L}_{b,c})$ and $c_1(\mathcal{U}_{b,c})$ (up to the fact that we have to pull back classes from $S^n C\times S^m C$ to $S^n C \times S^m C \times \pic^d C$ is this term $\eta_{n,d}$.
\end{proof}

\begin{proposition}
Let $V_{b,c,d,g,n,m}=\pi_\star \mathcal{U}_{b,c}$ where $\pi:S^n C \times S^m C \times \pic^d C\to S^m C \times \pic^d C$. $V_{b,c,d,g,n,m}$ is a vector bundle. Furthermore, suppose $d=b(n+g-1)+cm$. Then 
    \begin{equation}\label{chV}
        \ch(V_{b,c,d,g,n,m})=b^g e^{-\frac{c^2}{b}\theta_m  -n c \xi_m-\frac{1}{b}\theta_{d} - \frac{c}{b}\eta_{m,d}}.\end{equation}
\end{proposition}

The proof is carried out using the same method as the proof of Thm. \ref{thm:mainqhff}. 

The first two terms in Eq.\ \eqref{chV} correspond to the transport on the quasihole moduli space, as before in Thm.\ \ref{thm:mainqhff}. The third term is the Hall conductance for the transport on $\pic^d(C)$ and we interpret the third term as the 
conductance of the quasiholes, having the charge $-c/b$ of the opposite sign compared to particles, as expected.

\section{The multilayer case, with multiple quasiholes types.}
\label{sec:multilayer}
A natural generalization to consider is to allow for multiple types of interacting charge carriers, as well as multiple types of quasiholes with different charges. 
\paragraph{Notations}
Let $N_{l}, N_{q} \in \mathbb{N}$. Throughout this section, $N_l$ will represent the number of layers of particles and $N_q$ will represent the number of types of quasiholes where a "type" of quasiholes $s$ is characterized by a collection of numbers $\lbrace C_{is}, 1\leq i \leq N_l\rbrace$ of vanishing orders of the wave-functions when a particle of a given layer $i$ has a position equal to that of the quasihole of type $s$. The collection of these coefficients form a matrix $C\in M^{N_l,{N_q}}(\mathbb{N})$.

Let $n=(n_1,\dots,n_{N_l})^T$ (resp. $\vec{m}=(m_1,\dots,m_{N_q})^T$) be a vector whose components are the number of particles in each layer (resp. number of quasiholes of each type).

In this section, for a square matrix $A$, we denote by $\vec{A}$ the vector whose components are the diagonal entries of $A$. We also write $|A|$ the sum of the coefficients of $A$. 

Let $K\in M^{N_l\times N_l}(\mathbb{N})$ be a symmetric matrix.

\subsection{Construction of the multilayer quasihole bundle and computation of its first Chern class}
\label{subsec:constructionfirstchernclassmultilayer}

Let $L\rightarrow C$ be a degree $d$ holomorphic line bundle. We can build a line bundle over the Cartesian product $C^{\sum_i n_i}$ as $L^{\boxtimes \sum_i n_i}$. 
Let $$w=(w_1^1,\dots,w_{m_1}^1,\dots, w_1^{N_q},\dots,w_{m_{N_q}}^{N_q} )$$ be a point in $C^{\sum_s m_s}$.

Wavefunctions which, for each $1\leq s \leq {N_q}$ have localized quasiholes of type $s$ at positions $w_1^j,\dots,w_{m_j}^j$ are sections of $L^{\boxtimes \sum_i n_i}$ which satisfy the following axioms:

\begin{itemize}
    \item Their restriction to each copy of $C$ gives a holomorphic section of $L$.
    \item They are symmetric (resp. antisymmetric) under exchange of two particles inside the same layer $i$ when $K_{ii}$ is even (resp. odd)
    \item They vanish at order $K_{ij}$ when two particles in layers $i$ and $j$ are in the same position.
    \item  They vanish at order $C_{is}$ when a particle in layer $i$ is at position $w_\gamma^s$ for some $\gamma$.
\end{itemize}

Similarly to the one layer and one quasihole type computation, the vanishing condition is divisorial. 
For $1 \leq i,j\leq {N_l}$, let
$$\Delta_{ij} := \bigcup_{\substack{1 \le \mu \le n_i \\ 1 \le \nu \le n_j \\ (\mu<\nu \text{ if } i=j)}} \{ z^i_\mu = z^j_\nu \}$$ and 
$$ W_{i,w^s_1,\dots,w^s_{m_s}}:=\bigcup_{\substack{1 \le \mu \le n_i \\ 1 \le \gamma \le m_s}} \{ z^i_\mu = w^s_\gamma \}$$

We write $\Delta_K=\bigcup K_{ij}\Delta_{ij}$ and $\Delta_{C,w} :=\sum_{1\leq i\leq {N_l},\, 1\leq s \leq N_{q}} C_{is} W_{i, w^s_1,\dots,w^s_{m_s} }$.

To impose the vanishing conditions, we consider 

\begin{equation}
    L^{\boxtimes \sum_i n_i}(-\Delta_K-\Delta_{C,w})
\end{equation}

On $C^{\sum_i n_i}=\prod_i C^{n_i}$, we have the blockwise action of $\prod_i \mathfrak{S}_{n_i}$. Since $L^{\boxtimes \sum_i n_i}$, $\Delta_K$, and $\Delta_{C,w}$ are all invariant under this action, the line bundle constructed above descends to a unique line bundle $L_{K,C,w}$ on $\prod_i S^{n_i} C$. By construction, multilayer quantum Hall states with quasiholes are the holomorphic section of $L_{K,C,w}$.

These form a family of states parameterized by $\prod_s S^{m_s} C$ which will form a vector bundle. We will construct this vector bundle $V_{K,C}$ as a pushforward of a universal bundle $\mathcal{L}_{K,C}$ above $\prod_i S^{n_i} C \times \prod_s S^{m_s} C$. To construct this universal line bundle $\mathcal{L}_{K,C}$, note that $\Delta_{C,w}$ is clearly the restriction to $C^{\sum_i n_i}\times \{w\} $ of a divisor $\Delta_C$ on $C^{\sum_i n_i+\sum_s m_s} $, namely the divisor defined by $$\sum_{\substack{1\leq i\leq {N_l}\\ 1\leq s \leq {N_q}}} C_{is} \bigcup_{\substack{1 \le \mu \le n_i \\ 1 \le \gamma \le m_s}} \{ z^i_\mu = w^s_\gamma \}.$$

On $\prod_i C^{n_i} \times \prod_s C^{m_s}$, $\Delta_C$ as well as the pullbacks of $L^{\boxtimes \sum_i n_i}$ and $\Delta_K$ are left invariant by the block-wise action of $\prod_i \mathfrak{S}_{n_i} \times  \prod_s \mathfrak{S}_{m_s}$. As a consequence, $$L^{\boxtimes \sum_i n_i}(-\Delta_K-\Delta_C)$$ descends to a unique line bundle $\mathcal{L}_{K,C}$. By construction, for $w$ a point in $\prod_s S^{m_s} C$, $$\mathcal{L}_{K,C}|_{\prod_i S^{n_i} C \times \lbrace w \rbrace}=L_{K,C,w}.$$

\paragraph{Computation of the first Chern class}

In $\prod_{i=1}^{N_l} S^{n_i} C \times\prod_{s=1}^{N_l} S^{m_s}C$ we have the natural cohomology classes

\begin{itemize}
    \item $\theta_{n_i}$ (resp. $\theta_{m_s}$) be the theta class pulled back from  $\pic^{n_i}C$ (resp. $\pic^{m_s}C$)
    \item  $\eta_{n_i,n_j}$ (resp. $\eta_{m_s,m_t}$) the mixed class pulled back from $\pic^{n_i} \times \pic^{n_jC}$ (resp. $\pic^{m_s} \times \pic^{m_t}C$)
    \item  $\eta_{n_i,m_s}$ is the mixed class pulled back from $\pic^{n_i} \times \pic^{m_s}C$.
    \item $\xi_{n_i}$ (resp. $\xi_{m_s}$) be the $\xi$ class pulled back from $S^{n_i} C$ (resp. $S^{m_s} C$)
\end{itemize}

where mixed class $\eta_{i,j}$ means the natural mixed class that appears when pulling back the theta divisor $\theta_{i+j}$ in $\pic^{i+j} C$ via $\pic^{i} C \times \pic^j C \to \pic^{i+j} C$.

\begin{proposition}
We have 
$$c_1(\mathcal{L}_{K,C})=\sum_{i\in [{N_l}]} (K_{ii}\theta_{n_i} + p_i \xi_{n_i}) + \sum_{i<j,i,j\in [{N_l}]} K_{ij}\eta_{n_i,n_j} - \sum_{s\in [{N_q}]} (\vec{n}^T C)_s \xi_{m_s} + \sum_{i\in [{N_l}], s\in [{N_q}]} C_{is}\eta_{n_i,m_s}$$
where the $p_{n_i}$'s are such that, writing $\vec{p}^T=(p_{n_i})_{1\leq i \leq {N_l}}$ and $\vec{K}^T=(K_{ii})_{1\leq i \leq {N_l}}$ we have: 

\begin{equation}
\label{eq:configfirst}
\vec{d} - \vec{p} = K \vec{n} + C \vec{m} + \vec{K} (g-1).
\end{equation}
In this relation, $\vec{d}$ is a column vector of size ${N_l}$ whose components are the degree of the line bundle $L$.

\begin{proof}
Before tensoring by the divisor encoding the vanishing along diagonals, the line bundle whose pullback to $C^{\sum_i n_i+\sum_s m_s}$ is $L^{\boxtimes \sum_i n_i}$ (where we omit the pullback from $C^{\sum_i n_i}$ to $C^{\sum_i n_i+\sum_s m_s}$) has chern class $d\sum_i \xi_i$. We get the result tensoring this line bundle by $O(-D)$ with $D$ the divisor whose pullback on $C^{\sum_i n_i+\sum_s m_s}$ is $\Delta_K+\Delta_C$.
\end{proof}

\end{proposition}

\subsection{Grothendieck-Riemann Roch and Kodaira vanishing}
Let
$$\pi: \prod_{i=1}^{N_l} S^{n_i} C \times\prod_{s=1}^{N_q} S^{m_s}C \overset{\pi_1}{\to} \prod_{i=1}^{N_l} \pic^{n_i} C \times\prod_{s=1}^{N_q} S^{m_s}C \overset{\pi_2}{\to} \prod_{s=1}^{N_q} S^{m_s}C$$

The Grothendieck-Riemann-Roch theorem on the proper map $\pi$ reads 
$$\ch\left(\sum_i (-1)^iR^i \pi_\star \mathcal{L}_{K,C}\right) \td \left(\prod_{s\in [{N_q}]} S^{m_s}C\right) =\pi_\star \left(e^{c_1(\mathcal{L}_{K,C})} \td\left(\prod_{i\in [{N_l}]} S^{n_i}C \times \prod_{s\in [{N_q}]} S^{m_s}C\right)\right).$$ Using the projection formula and the multiplicativity of the todd class we get
$$\ch\left(\sum_i (-1)^iR^i \pi_\star \mathcal{L}_{K,C}\right) = \pi_\star \left(e^{c_1(\mathcal{L}_{K,C})} \td\left(\prod_{i\in [{N_l}]} S^{n_i}C \right)\right)$$ 

We will show using the Kodaira vanishing theorem that all the higher direct images vanish and that the left hand side is $\ch(V)$ where $V=R^0 \pi_\star \mathcal{L}_{K,C}$.

\begin{proposition}
\label{prop:kodairavanishing}
Suppose $K-I$ is non-negative and for all $i$, $p_i>-(n_i+1-g)$. Then the Kodaira vanishing theorem applies and $\ch\left(\sum_i (-1)^iR^i \pi_\star \mathcal{L}_{K,C}\right)=\ch\left( R^0 \pi_\star \mathcal{L}_{K,C}\right)$. We write $V=R^0 \pi_\star \mathcal{L}_{K,C}$.    
\end{proposition}

\begin{proof}
Let $y \in \prod_s S^{m_s}C$ and $\mathcal{L}_{K,C}={\mathcal{L}_{K,C}}_{|y}$. Let $(K,C,d,g,n,m)$ be a configuration as in the statement. We write $\omega_X$ the canonical line bundle on $X=\prod_{i=1}^{{N_l}}S^{n_i}C$.
The result will follow from the Kodaira vanishing theorem if we show that $L_{K,C} \otimes {\omega_X}^{-1}$ is ample. We know that $\omega_{S^NC}=\theta_N-(N+1-g)\xi_N$, thus 
$$\omega_X=\sum_i \theta_{n_i}-(n_i+1-g)\xi_{n_i}.$$
By the results of subsection \ref{subsec:constructionfirstchernclassmultilayer}, $L_{K,C} \otimes \omega_X^{-1}$ has first Chern class

$$\sum_{i} (K_{ii}  \theta_{n_i}  + p_{n_i} \xi_{n_i} )+\sum_{i<j} K_{ij} \eta_{n_i,n_j} - \sum_i \left(\theta_{n_i} -(n_i+1-g) \xi_{n_i} \right)$$ which can be rewritten
$$\sum_{i}(p_{n_i}+(n_i+1-g)) \xi_{n_i} + \sum_{i<j} (K-I)_{ij} \eta_{n_i,n_j} + \sum_i (K-I)_{ii} \theta_{n_i} .$$ By \cite[Prop. 4.1.3]{Aldonza_Dupont_2025}, this class is ample.

\end{proof}

\begin{corollary}
\label{corr:kodaira}
    Let $(K,C,d,g,n,m)$ be a configuration as in proposition \ref{prop:kodairavanishing}. From the latter proposition as well as Grauert theorem, we find that $\sum_i (-1)^iR^i \pi_\star \mathcal{L}_{K,C}=R^0 \pi_\star \mathcal{L}_{K,C} :=V_{K,C,g,d,n,m} $ is a vector bundle and that the l.h.s. of the Grothendieck-Riemann-Roch formula compute its Chern character: $\ch\left(\sum_i (-1)^iR^i \pi_\star \mathcal{L}_{K,C}\right)=\ch(V_{K,C,g,d,n,m})$.
\end{corollary}

\subsection{Computation of the pushforward}

\label{sec:symplectic-basis}
	
    As previously, we introduce symplectic bases $\lbrace \alpha_{n_i}^l,\beta_{n_i}^l,1\leq l\leq g \rbrace$ and $\lbrace {\alpha}_{m_s}^l,{\beta}_{m_s}^l,1\leq l\leq g \rbrace$ of each $H^1(\pic^{n_i}(C),\mathbb{Z})$ and $H^1(\pic^{m_s}(C),\mathbb{Z})$. With these notations, the elements of $NS( \prod_{i=1}^{N_l} S^{n_i} C \times\prod_{s=1}^{N_q} S^{m_s}C)$ we introduced can be represented as

    $$\theta_{n_i}=\sum_{l=1}^g \alpha_{n_i}^l \wedge \beta_{n_i}^l \text{ (resp. }\theta_{m_s}=\sum_{l=1}^g {\alpha}_{m_s}^l \wedge \beta_{m_s}^l)$$
	$$\eta_{n_i,n_j}=\sum_{l=1}^g (\alpha_{n_i}^l \wedge\beta_{n_j}^l +\alpha_{n_j}^l \wedge\beta_{n_i}^l) \text{ (resp. } \eta_{m_s,m_t}=\sum_{l=1}^g ({\alpha}_{m_s}^l \wedge {\beta}_{m_t}^l +{\alpha}_{m_t}^l \wedge {\beta}_{m_s}^l)$$
 	$$\eta_{n_i,m_s}=\sum_{l=1}^g (\alpha_{n_i}^l \wedge {\beta}_{m_s}^l +{\alpha}_{m_s}^l \wedge\beta_{n_i}^l)$$

\begin{lemma}
Suppose $\vec{d}=K\vec{n}+C\vec{m}+\vec{K}(g-1)$. Then 

$${\pi_1}_\star \left(e^{c_1(\mathcal{L}_{K,C})} \td\left(\prod_{i\in [{N_l}]} S^{n_i}C \right)\right)=e^{c_1(\mathcal{L}_{K,C})}$$
    
\end{lemma}

\begin{proof}
Notice that when $\vec{d}=K\vec{n}+C\vec{m}+\vec{K}(g-1)$, the cohomology class $c_1(\mathcal{L}_{K,C})$ is a pullback via $\pi_1$ of a class on $\prod_i \pic^{n_i} C \times \prod_s S^{m_s} C$, thus the projection formula gives 
\eqtag{{\pi_1}_\star \left(e^{c_1(\mathcal{L}_{K,C})} \td\left(\prod_{i\in [{N_l}]} S^{n_i}C \right)\right)=e^{c_1(\mathcal{L}_{K,C})} {\pi_1}_\star \left( \td\left(\prod_{i\in [{N_l}]} S^{n_i}C \right)\right).}{eq:pushforwardp1simple} 
The Todd class is multiplicative and $\pi_1$ acts on each factor independently. By lemma \ref{lem:firstpushforward_lemma1}, we have 
$${\pi_1}_\star \left( \td\left(\prod_{i\in [{N_l}]} S^{n_i}C \right)\right)=\prod_{i\in [{N_l}]}  {\pi_1}_\star \left( \td\left(S^{n_i}C \right)\right)=1
$$ where we omit the write pullbacks and in each factor of the product we denoted by${\pi_1}_\star$ the projection on a single factor $S^{n_i}C\to \pic^{n_i} C$.
\end{proof}

Let $\Theta_m$ be a matrix of differential forms of size ${N_q}\times {N_q}$ defined by 

\begin{equation}
\label{eq:bigtheta}
\Theta_m=\left( \sum_l \left( {\alpha}^l_{m_s} \wedge {\beta}^l_{m_t} \right) \right)_{1\leq s,t \leq {N_q}}
\end{equation}
and $\xi_m$ the vector 
\begin{equation}
\label{eq:xivect}
\xi_m=\left(\xi_{m_s}\right)_{1\leq s\leq {N_q}}.
\end{equation}

\mainqhmulti*

\begin{proof}
What is left is to compute ${\pi_2}_\star$ which corresponds to integration on $\prod_{i=1}^{N_l} \pic^{n_i} C$. It selects only the volume form of $\prod_{i=1}^{N_l} \pic^{n_i} C$. This is a combinatorial problem, we can forget that the $\psi_i^r$,$\ov{\psi^s_j}$ are differential forms and simply see them as generators of a so called Grassmann algebra. For more details about Grassmann algebras, see \cite{Berezin_1987,Caracciolo_Sokal_Sportiello_2013}. Here, we recall that for $R$ a commutative ring with identity and $m$ symbols $\chi_a, 1 \leq a \leq m$, the associated Grassmann algebra is the algebra generated by the $\chi_a$ quotiented by the relations $\chi_a \chi_b + \chi_b \chi_a = 0$ for $1\leq a < b \leq m$ and $\chi_a^2=0, 1\leq a \leq m$. On such an algebra we can define an $R$-linear operator $\int d \chi_a$ as 
 $$\int d\chi_a \chi_{a_1}\dots \chi_{a_q}=
\begin{cases}
    (-1)^{\delta-1} \chi_{a_1}\dots \chi_{a_{\delta-1}},\chi_{a_{\delta+1}}\dots, \chi_{a_q} &\text{if } a = a_\delta \\
    0 & \text{if } a \notin \lbrace a_1,\dots, a_q\rbrace.
\end{cases}
$$ for any monomial $\chi_{a_1}\dots \chi_{a_q}$ with $a_1 <\dots < a_q$. A Grassmann algebra is a graded algebra, with $\deg\left(\chi_{a_1}\dots \chi_{a_q}\right)=q$.

In our case, we introduce the Grassmann algebra generated by the symbols $\psi^l=(\beta_{n_i}^l)_{1\leq i \leq {N_l}}$, $\ov{\psi^l}=(\alpha_{n_i}^l)_{1\leq i \leq {N_l}}$, $\phi^l=({\beta}_{m_s}^l)_{1\leq s \leq {N_q}}$ and $\ov{\phi^l}=({\alpha}_{m_s}^l)_{1\leq s \leq {N_q}}$. As an example, fix $i\in[k],r\in[g]$ and $\kappa$ an element of the Grassman algebra which does not contain $\psi_i^r$. Then for any $j \in[k],s\in[g]$ we have $$\int d{\psi^r_i} \psi^s_{j} \kappa  = \kappa \delta_i^j \delta_r^s$$ while $$\int d{\psi^r_i} \kappa \psi^s_{j}   = (-1)^{\deg(\kappa)} \kappa \delta_i^j \delta_r^s$$
where $\delta_i^j$ is the Kronecker symbol. Note that to shorten the notation, we omit the $\wedge$ symbol when multiplying forms. With these variables, 

$$c_1(\mathcal{L}_{K,C})=\sum_{1\leq l \leq g} \left( \ov{\psi^l}^T K \psi^l - \ov{\phi^l}^T C^T \psi^l - \ov{\psi^l}^T C \phi^l \right) - \vec{n}^T C \xi_m.$$ By the projection formula, 

$${\pi_2}_\star \left( e^{\sum_{l=1}^g \left( \ov{\psi^l}^T K \psi^l - \ov{\phi^l}^T C^T \psi^l - \ov{\psi^l}^T C \phi \right) - \vec{n}^T C \xi_m}\right)=e^{- \vec{n}^T C \xi_m} {\pi_2}_\star \left( e^{\sum_{l=1}^g \left( \ov{\psi^l}^T K \psi^l - \ov{\phi^l}^T C^T \psi^l - \ov{\psi^l}^T C \phi^l \right) } \right).$$

Note that in $\prod_{i=1}^{k} \prod_{r=1}^g \ov{\psi^r_i} \psi^r_i $, the ordering in $r$ and $i$ does not matter because each element of the form $\ov{\psi^r_i} {\psi^r}_i$ is even in the Grassmann algebra.

In our computation of the pushforward under $\pi_1$, we regrouped terms when they belonged to the same layer $S^{n_i} C$, for $1\leq i \leq k$. 
Yet, in the notations we introduced, the quantity we want to pushforward takes the form of a product over upper indices $1\leq r \leq g$, so we will carry out the extraction as 
$${\pi_2}_\star= \int \prod_{r=1}^g \left[ D(\psi^r,\ov{\psi^r}) \right].$$ 
where $D(\psi^r,\ov{\psi^r})= \prod_{i=1}^k d\psi^r_i d\ov{\psi^r}_i$.
We denote by $\int d {\psi^l}_{i}$ (resp. $\int d \ov{\psi^l}_i$) the operator that selects the terms containing ${\psi^l}_{i}$ (resp. $\ov{\psi^l}_i$) and integrate it out. With those notations, $${\pi_2}_\star= \int \prod_{l=1}^g \left[ D(\psi^l,\ov{\psi^l}) \right]$$ where we introduced $D(\psi^l,\ov{\psi^l})= \prod_{i=1}^{N_l} d\psi^l_i d\ov{\psi^l}_i$. We regroup terms $d\psi^l_i d\ov{\psi^l}_i$ having the same upper index $l$ because the expression from which we extract coefficients has itself the form of a product which separates variables.

By Wick’s theorem for "complex fermions" (see \cite{Caracciolo_Sokal_Sportiello_2013}[Thm. A.16]), we have for each $l$ fixed

$${\pi_2}_\star\left( e^{ \ov{\psi^l}^T K \psi^l - \ov{\phi^l}^T C^T \psi^l - \ov{\psi^l}^T C \phi }\right)= \det(K) e^{-\ov{\phi^l}^T C^T K^{-1} C \phi^l}.$$ It follows that
$$\ch(V_{K,C,d,g,\vec{n},\vec{m}})=\prod_{l=1}^g \left( \det(K) e^{-\ov{\phi^l}^T C^T K^{-1} C \phi^l} \right) e^{- \vec{n}^T C \xi_m}.$$
The result is then obtained by using $(\Theta_m)_{s,t}=\sum_l {\ov{\phi^l}_s} {\phi^l_t}$.

\end{proof}

\printbibliography

\end{document}